\documentclass[11pt,a4paper,reqno]{amsart}

\usepackage{amssymb,amsmath}
\usepackage{latexsym}
\usepackage{hyperref}
\usepackage{graphics}
\usepackage{euscript}
\usepackage{mathrsfs}
\usepackage[dvips]{curves}
\usepackage{tikz}
\usepackage{pgf}
\usetikzlibrary{arrows,decorations.pathmorphing,backgrounds,positioning,fit}
\usepackage{epsf}
\usepackage{xcolor}
\usepackage[percent]{overpic}
\usepackage{setspace}
\usepackage{multirow}
\usepackage{comment}

\def\sqbullet{\raise.2ex\hbox{\vrule width 3.5pt height 3.5pt}}

\hypersetup{
 colorlinks,
 citecolor=blue,
 filecolor=black,
 linkcolor=blue,
 urlcolor=black
}

\def\sqbullet{\raise.2ex\hbox{\vrule width 3.5pt height 3.5pt}}

{\em}

\newcounter{substep}
\def\thesubstep{\arabic{substep}}

\newenvironment{substeps}[1]{%
\refstepcounter{substep}\noindent{(\ref{#1}.\thesubstep)\ }\ }%
{\em}

\newcounter{subsubstep}
\def\thesubsubstep{\arabic{subsubstep}}

{\em}

\numberwithin{figure}{section}

\newtheorem{mthm}{Theorem}
\newtheorem{thm}{Theorem}[section]

\newtheorem{prop}[thm]{Proposition}
\newtheorem{cor}[thm]{Corollary}
\newtheorem{lem}[thm]{Lemma}

\theoremstyle{definition}

\newtheorem{define}[thm]{Definition}
\newtheorem{defines}[thm]{Definitions}

\theoremstyle{remark}
\newtheorem{remark}[thm]{Remark}
\newtheorem{remarks}[thm]{Remarks}

 \newcommand{\N}{{\mathbb N}}
 \newcommand{\R}{{\mathbb R}}

\newcommand{\sph}{{\mathbb S}}

\newcommand{\pol}{{\EuScript K}}

\newcommand{\p}{{\EuScript P}}

\newcommand{\Ss}{{\EuScript S}}

\newcommand{\Nn}{{\EuScript N}}
\newcommand{\Qq}{{\EuScript Q}}

\newcommand\Rr{{\EuScript R}}
\newcommand{\Bb}{{\EuScript B}}

\newcommand{\Mm}{{\EuScript M}}
\newcommand{\Tt}{{\EuScript T}}
\newcommand{\Hh}{{\EuScript H}}
\newcommand{\Ii}{{\EuScript I}}

\newcommand{\Ff}{{\EuScript F}}

\newcommand{\Cc}{{\EuScript C}}

\newcommand{\Ee}{{\EuScript E}}

\newcommand{\tildebaja}{{\raise.17ex\hbox{$\scriptstyle\sim$}}}

\newcommand{\Int}{\operatorname{Int}}

\newcommand{\pp}{\operatorname{p}}
\newcommand{\rr}{\operatorname{r}}
\newcommand{\TF}{\operatorname{\mathsf F}}
\newcommand{\TS}{\operatorname{\mathsf G}}

\newcommand{\id}{\operatorname{id}}

\newcommand{\x}{{\tt x}} 
 \renewcommand{\t}{{\tt t}}

\newcommand{\ol}{\overline}

\newcommand{\veps}{\varepsilon}
\newcommand{\bs}{{\mathcal B}}

\title[On complements of convex polyhedra as polynomial and regular images]{On complements of convex polyhedra\\ as polynomial and regular images of $\R^n$}

\date{}

\author{Jos\'e F. Fernando}
\address{Departamento de \'Algebra, Facultad de Ciencias Matem\'aticas, Universidad Complutense de Madrid, 28040 MADRID (SPAIN)}
\curraddr{}
\email{josefer@mat.ucm.es}\thanks{The first author is supported by Spanish GR MTM2011-22435, while the second is a external collaborator of this project. This article is based on a part of the doctoral dissertation of the second author, written under the supervision of the first. 
}

\author{Carlos Ueno}
\address{Departamento de Matem\'aticas, IES La Vega de San Jos\'e, Paseo de San Jos\'e, s/n, Las Palmas de Gran Canaria, 35015 LAS PALMAS (SPAIN).
}
\email{carlos.ueno@terra.es}
\begin{document}

\begin{abstract}
In this work we prove constructively that the complement $\R^n\setminus\pol$ of a convex polyhedron $\pol\subset\R^n$ and the complement $\R^n\setminus\Int(\pol)$ of its interior are regular images of $\R^n$. If $\pol$ is moreover bounded, we can assure that $\R^n\setminus\pol$ and $\R^n\setminus\Int(\pol)$ are also polynomial images of $\R^n$. The construction of such regular and polynomial maps is done by double induction on the number of \em facets \em (faces of maximal dimension) and the dimension of $\pol$; the careful placing (\em first \em and \em second trimming positions\em) of the involved convex polyhedra which appear in each inductive step has interest by its own and it is the crucial part of our technique.
\end{abstract}

\date{31/08/2012}

\maketitle

\section*{Introduction}\label{s1}

A map $f:=(f_1,\ldots,f_m):\R^n\to\R^m$ is a \em polynomial map \em if its components $f_k\in\R[\x]:=\R[\x_1,\ldots,\x_n]$ are polynomials. Analogously, $f$ is a \em regular map \em if its components can be represented as quotients $f_k=g_k/h_k$ of two polynomials $g_k,h_k\in\R[\x]$ such that $h_k$ never vanishes on $\R^n$. During the last decade we have approached the problem of determining which (semialgebraic) subsets $\Ss\subset\R^m$ are polynomial or regular images of $\R^n$ (see \cite{g} for the first proposal of studying this problem and some particular related ones like the ``quadrant problem''). By Tarski-Seidenberg's principle (see \cite[1.4]{bcr}) the image of an either polynomial or regular map is a semialgebraic set. As it is well-known a subset $\Ss\subset\R^n$ is \em semialgebraic \em when it has a description by a finite boolean combination of polynomial equations and inequalities, which we will call a \em semialgebraic \em description.

We feel very far from solving the problem stated above in its full generality, but we have developed significant progresses in two ways:

\noindent{\em General conditions.} By obtaining general conditions that must satisfy a semialgebraic subset $\Ss\subset\R^m$ which is either a polynomial or a regular image of $\R^n$ (see \cite{fg2,fu,u1} for further details). The most remarkable one states that the set of infinite points of a polynomial image of $\R^n$ is connected.

\noindent{\em Ample families.} By showing that certain ample families of significant semialgebraic sets are either polynomial or regular images of $\R^n$ (see \cite{f1,fg1,fgu1,u2} for further details). We will comment later some of our main results.

The results of this work go in the direction of the second approach. Since the $1$-dimensional case has been completely described by the first author in \cite{f1}, we will only take care of semialgebraic sets $\Ss\subset\R^m$ of dimension $\geq2$.

A very distinguished family of semialgebraic sets is the one consisting of those semialgebraic sets whose boundary is piecewise linear, that is, semialgebraic sets which admit a semialgebraic description involving just polynomials of degree one. Of course, many of them are  neither polynomial nor regular images of $\R^n$, as one easily realizes reviewing the known general conditions that must satisfy a semialgebraic set to be either a polynomial or a regular image of $\R^n$; however, it seems natural to wonder what happens with convex polyhedra, their interiors (as topological manifolds with boundary), their complements and the complements of their interiors.

We proved in \cite{fgu1} that all $n$-dimensional convex polyhedra of dimension $\geq2$ and their interiors are regular images of $\R^n$. Since many convex polyhedra are bounded and the images of nonconstant polynomial maps are unbounded, the suitable general approach in that case was to consider regular maps. Concerning the representation of unbounded convex polygons $\p\subset\R^2$ as polynomial images of $\R^2$, see \cite{u3}. 

The complements of proper convex polyhedra (or of their interiors) are unbounded and so the representation problem in these cases has a meaning either in the polynomial or in the regular case. Our main result is the following:

\begin{mthm}\label{mainext}
Let $\pol\subset\R^n$ be a convex polyhedron such that $\pol$ is not a layer nor a hyperplane. Then  
\begin{itemize}
\item[(i)] If $\pol$ is bounded, $\R^n\setminus\pol$ and $\R^n\setminus\Int(\pol)$ are polynomial images of $\R^n$. 
\item[(ii)] If $\pol$ is unbounded, $\R^n\setminus\pol$ and $\R^n\setminus\Int(\pol)$ are regular images of $\R^n$. 
\end{itemize}
\end{mthm}

Here, $\Int(\pol)$ refers to the interior of $\pol$ as a topological manifold with boundary, and by a \em layer \em we understand a subset of $\R^n$ affinely equivalent to $[-a,a]\times\R^{n-1}$ with $a>0$. It is important to point out that the challenging situations arise when the dimension of $\pol\subset\R^n$ equals $n$. As we will see in Proposition \ref{semial} and Remark \ref{semialr}, for convex polyhedra $\pol\subset\R^n$ of dimension $d<n$ which are not hyperplanes it is always true (and much easier to prove) that both $\R^n\setminus\pol$ and $\R^n\setminus\Int(\pol)$ are polynomial images of $\R^n$. 

It is clear that layers and hyperplanes disconnect $\R^n$. Actually, from Theorem~\ref{mainext} follows as a byproduct that these are the only convex polyhedra that disconnect $\R^n$ (fact that can be proved by more elementary means), for a polynomial or  regular image of $\R^n$ must be necessarily connected.

Moreover, in \cite{fu2} we prove that the techniques developed here still work, by carefully handling our convex polyhedra, to prove that $\R^3\setminus\pol$ and $\R^3\setminus\Int(\pol)$ are polynomial images of $\R^3$ if $\pol\subset\R^3$ is a $3$-dimensional convex polyhedron; even more, we prove there that for $n\geq4$ our constructions do not go further to represent either $\R^n\setminus\pol$ or $\R^n\setminus\Int(\pol)$ as polynomial images of $\R^n$, for a general convex polyhedra $\pol\subset\R^n$. 

To easy the presentation of the full picture of what is known concerning piecewise linear boundary sets (including the results of this work stated in Theorem \ref{mainext}) we introduce the following two invariants. Given a semialgebraic set $\Ss\subset\R^m$ we denote 
\begin{equation*}
\begin{split}
\pp(\Ss):&=\inf\{n\geq1:\exists \ f:\R^n\to\R^m\, \, \, \text{polynomial such that}\, \, \, f(\R^n)=\Ss\},\\
\rr(\Ss):&=\inf\{n\geq1:\exists \ f:\R^n\to\R^m\, \, \, \text{regular such that}\, \, \, f(\R^n)=\Ss\}.
\end{split}
\end{equation*}
The condition $\pp(\Ss):=+\infty$ characterizes the nonrepresentability of $\Ss$ as a polynomial image of some $\R^n$, while $\rr(\Ss):=+\infty$ has the analogous meaning for regular maps. Let $\pol\subset\R^n$ be an $n$-dimensional convex polyhedron and assume for the second part of the table below that $\Ss:=\R^n\setminus\pol$ and $\ol{\Ss}=\R^n\setminus\Int(\pol)$ are connected:

\begin{center}
{\setlength{\arrayrulewidth}{.5pt}
\renewcommand*{\arraystretch}{1.2}
\begin{tabular}{|c|c|c|c|c|c|c|c|c|c|}
\hline
&\multicolumn{2}{c|}{$\pol$ bounded}&\multicolumn{4}{c|}{$\pol$ unbounded}\\
\cline{2-7} 
&$n=1$&$n\geq2$&$n=1$&$n=2$&$n=3$&$n\geq4$\\
\hline
${\rm r}(\pol)$&\multicolumn{6}{c|}{$n$}\\
\hline
${\rm r}(\Int(\pol))$&\multirow{3}{*}{$+\infty$}&$n$&2&\multicolumn{3}{c|}{$n$}\\
\cline{1-1}\cline{3-7}
${\rm p}(\pol)$&&\multirow{2}{*}{$+\infty$}&1&$2,+\infty^*$&\multicolumn{2}{c|}{unknown$_1$, $+\infty^*$}\\
\cline{1-1}\cline{4-7}
${\rm p}(\Int(\pol))$&&&2&$+\infty,\,2^*$&\multicolumn{2}{c|}{unknown$_2$, $+\infty^*$}\\
\hline
${\rm r}(\Ss)$&\multirow{4}{*}{$+\infty$}&\multirow{4}{*}{$n$}&2&\multicolumn{3}{c|}{$n$}\\
\cline{1-1}\cline{4-7}
${\rm r}(\ol{\Ss})$&&&\multicolumn{4}{c|}{$n$}\\
\cline{1-1}\cline{4-7}
${\rm p}(\Ss)$&&&2&\multicolumn{2}{c|}{$n$}&\multirow{2}{*}{unknown}\\
\cline{1-1}\cline{4-6} 
${\rm p}(\ol{\Ss})$&&&\multicolumn{3}{c|}{$n$}&\\
\hline
\end{tabular}
}
\end{center}

\noindent Let us explain some (marked) exceptions in the previous chart:
\begin{itemize}
\item[$\sqbullet$] $(2,\,+\infty^*)$: an unbounded convex polygon $\pol\subset\R^2$ has ${\rm p}(\pol)=+\infty$ if and only if all the unbounded edges of $\pol$ are parallel  (see \cite{u3}); otherwise ${\rm p}(\pol)=2$.
\item[$\sqbullet$] $(+\infty,\,2^*)$: if the number of edges of an unbounded convex polygon $\pol$ is $\geq3$, then ${\rm p}(\Int(\pol))=+\infty$; otherwise, ${\rm p}(\Int(\pol))=2$ (see \cite{fg2,u2}).
\item[$\sqbullet$] (unknown$_1$, $+\infty^*$): if the linear projection of $\pol$ in some direction is bounded, then ${\rm p}(\Int(\pol))=+\infty$; otherwise, we have no further information.
\item[$\sqbullet$] (unknown$_2$, $+\infty^*$): if $\pol$ has a bounded facet or its linear projection in some direction is bounded, then ${\rm p}(\Int(\pol))=+\infty$; otherwise, we have no further information.
\end{itemize}

The interest of deciding whether a semialgebraic set is a polynomial or a regular image of $\R^n$ is out of any doubt, and it lies in the fact that the study of certain classical problems in Real Geometry concerning this kind of sets is reduced to the analysis of those problems on $\R^n$, for which no contour conditions appear and there are many more tools and more powerful (for further details see \cite[\S1]{fg1} and \cite[\S1]{fg2}). Let us discuss here two of them. 

\noindent{\em Optimization}. Suppose that $f:\R^n\to\R^n$ is either a polynomial or a regular map and let $\Ss=f(\R^n)$. Then the optimization of a given regular function $g:\Ss\to\R$ is equivalent to the optimization of the composition $g\circ f$ on $\R^n$, and in this way one can avoid contour conditions (see for instance \cite{nds,ps,sch} for relevant tools concerning optimization of polynomial functions on $\R^n$). 

\noindent{\em Positivstellens\"atze}. Another classical problem is the algebraic characterization of those regular functions $g:\R^n\to\R$ which are either strictly positive or positive semidefinite on $\Ss$. In case $\Ss$ is a basic closed semialgebraic set these problems were solved in \cite{s}, see also \cite[4.4.3]{bcr}. Note that $g$ is strictly positive (resp. positive semidefinite) on $\Ss$ if and only if $g\circ f$ is strictly positive (resp. positive semidefinite) on $\R^n$, and both questions are decidable using for instance \cite{s}. Thus, this provides an algebraic characterization of positiveness for polynomial and regular functions on semialgebraic sets which are either polynomial or regular images of $\R^n$. Observe that these semialgebraic sets need not to be either closed, as it happens with the interior of a convex polyhedron, or basic, as it happens with the complement of a convex polyhedron. Thus, this provides a large class of semialgebraic sets (neither closed nor basic), which are not considered by the classical Positivstellens\"atze, for which there is a certificate of positiveness for polynomial and regular functions.

The article is organized as follows. In Section \ref{s2} we provide some general terminology and results concerning convex polyhedra. In Section \ref{s3} we check Theorem~\ref{mainext} for those convex polyhedra in $\R^n$ whose dimension is strictly smaller than $n$; in fact, we prove a more general result which works for the complement of any basic semialgebraic set contained in a hyperplane. In Sections \ref{s4} and \ref{s6} we develop the trimming tools that will be crucial to prove Theorem \ref{mainext} in Sections \ref{s5} and \ref{s7}. Finally, in Section \ref{s8} we prove, as a byproduct of Theorem \ref{mainext}, that the complement of an open ball is a polynomial image of $\R^n$ while the complement of the closed ball is a polynomial image of $\R^{n+1}$ but not of $\R^n$.

\section{Preliminaries on convex polyhedra}\label{s2}

We begin by introducing some preliminary terminology and notations concerning convex polyhedra. For a detailed study of the main properties of convex sets, we refer the reader to \cite{ber1,ber2,r}. Along this work, we write $\R[\x]$ to denote the ring of polynomials $\R[\x_1,\dots,\x_n]$. Given an affine hyperplane $H\subset\R^n$ there exists a degree one polynomial $h\in\R[\x]$ such that $H:=\{x\in\R^n:\ h(x)=0\}\equiv\{h=0\}$. This polynomial $h$ determines two \emph{closed half-spaces} $H^+$ and $H^-$, defined by
$$
H^+:=\{x\in\R^n:\ h(x)\geq 0\}\equiv\{h\geq0\}\quad\text{and}\quad H^-:=\{x\in\R^n:\ h(x)\leq 0\}\equiv\{h\leq0\}.
$$
In this way, the labels $H^+$ and $H^-$ are assigned in terms of the choice of $h$ (or equivalently, in terms of the orientation of $H$) and it is enough to substitute $h$ by $-h$ to interchange the roles of $H^+$ and $H^-$. As usual, we denote by $\vec{h}$ the homogeneous component of $h$ of degree one and by $\vec{H}:=\{\vec{v}\in\R^n:\ \vec{h}(\vec{v})=0\}$ the \em direction of $H$\em. More generally, we also denote by $\vec{W}$ the direction of any affine subspace $W\subset\R^n$.

\subsection{Generalities on convex polyhedra}\label{pldrfc} 
A subset $\pol\subset\R^n$ is a \emph{convex polyhedron} if it can be described as the finite intersection $\pol:=\bigcap_{i=1}^rH_i^+$ of closed half-spaces $H_i^+$; we allow this family of half-spaces to by empty to describe $\pol=\R^n$ as a convex polyhedron. The dimension $\dim(\pol)$ of $\pol$ is its dimension as a topological manifold with boundary, which coincides with its dimension as a semialgebraic subset of $\R^n$ and with the dimension of the smallest affine subspace of $\R^n$ which contains $\pol$. 

Let $\pol\subset\R^n$ be an $n$-dimensional convex polyhedron. By \cite[12.1.5]{ber2} there exists a unique minimal family ${\mathfrak H}:=\{H_1,\ldots,H_m\}$ of affine hyperplanes of $\R^n$, which is empty just in case $\pol=\R^n$, such that $\pol=\bigcap_{i=1}^mH_i^+$; we refer to this family as the \em minimal presentation \em of $\pol$. Of course, we will always assume that we have chosen the linear equation $h_i$ of each $H_i$ such that $\pol\subset H_i^+$. 

The \em facets \em or $(n-1)$-\em faces \em of $\pol$ are the intersections $\Ff_i:=H_i\cap\pol$ for $1\leq i\leq m$; just the convex polyhedron $\R^n$ has no facets. Each facet $\Ff_i:=H_i^-\cap\bigcap_{j=1}^mH_j^+$ is a polyhedron contained in $H_i$. For inductive processes, we will denote $\pol_{i,\times}:=\bigcap_{j\neq i}H_j^+$, which is a convex polyhedron that strictly contains $\pol$ and satisfies $\pol=\pol_{i,\times}\cap H_i^+$. 

A convex polyhedron $\pol\subset\R^n$ is a topological manifold, whose interior $\Int(\pol)$ is a topological manifold without boundary, which coincides with the topological interior of $\pol$ in $\R^n$ if and only if $\dim(\pol)=n$. The \em boundary \em of $\pol$ is $\partial\pol=\pol\setminus\Int(\pol)$; in particular, if $\pol$ is a singleton, $\Int(\pol)=\pol$ and $\partial\pol=\varnothing$. By \cite[12.1.5]{ber2}, $\partial\pol=\bigcup_{i=1}^m\Ff_i$, and so
$$
\Int(\pol)=\pol\setminus\bigcup_{i=1}^m(\pol\cap H_i)=\bigcap_{j=1}^mH_j^+\cap\bigcap_{i=1}^m(\R^n\setminus H_i)=\bigcap_{i=1}^m(H_i^+\setminus H_i).
$$ 

For $0\leq j\leq n-2$, we define inductively the $j$-\em faces \em of $\pol$ as the facets of the $(j+1)$-faces of $\pol$, which are again convex polyhedra. As usual, the $0$-faces are the \em vertices \em of $\pol$ and the $1$-faces are the \em edges \em of $\pol$; obviously if $\pol$ has a vertex, then $m\geq n$ (see \cite[12.1.8-9]{ber2}). In general, we denote by $\Ee$ a generic face of $\pol$ to distinguish it from the facets $\Ff_1,\ldots,\Ff_m$ of $\pol$; the affine subspace generated by $\Ee$ is denoted by $W$ to be distinguished from the affine hyperplanes $H_1,\ldots,H_m$ generated by the facets $\Ff_1,\ldots,\Ff_m$ of $\pol$.

In what follows, given two points $p,q\in\R^n$ we denote by $\ol{pq}:=\{\lambda p+(1-\lambda)q:\ 0\leq\lambda\leq1\}$ the \em (closed) segment connecting $p$ and $q$\em. On the other hand, given a point $p\in\R^n$ and a vector $\vec{v}\in\R^n$, we denote by $p\vec{v}:=\{p+\lambda\vec{v}: \lambda\geq0\}$ the \em (closed) half-line of extreme $p$ and direction $\vec{v}$\em. The boundary of the segment $\Ss:=\ol{pq}$ is $\partial\Ss=\{p,q\}$ and its interior is $\Int(\Ss)=\Ss\setminus\{p,q\}$, while the boundary of the half-line $\Hh:=p\vec{v}$ is $\partial\Hh=\{p\}$ and its interior is $\Int(\Hh)=\Hh\setminus\{p\}$.

The \em affine maps \em are those maps that preserve affine combinations; hence, affine maps are polynomial maps that preserve convexity and so they transform convex polyhedra onto convex polyhedra. As usual, the linear map (or derivative) associated to an affine map $f:\R^n\to\R^m$ is denoted by $\vec{f}:\R^n\to\R^m$. It will be common to use affine bijections $f:\R^n\to\R^n$, that will be called \em changes of coordinates\em, to place a convex polyhedron, or more generally a semialgebraic set $\Ss$, in convenient positions. As usual we say that $\Ss$ and $f(\Ss)$ are \em affinely equivalent\em.

A convex polyhedron of $\R^n$ is \em nondegenerate \em if it has at least one vertex; otherwise, the convex polyhedron is \em degenerate\em. Next, we recall a natural characterization of degenerate convex polyhedra (see \cite[\S2.2]{fgu1}). 

\begin{lem}\label{lem:deg2} 
Let $\pol\subset\R^n$ be an $n$-dimensional convex polyhedron. The following assertions are equivalent:
\begin{itemize}
\item[(i)] $\pol$ is degenerate.
\item[(ii)] $\pol=\R^n$ or there exists a nondegenerate convex polyhedron $\p\subset\R^k$, where $1\leq k\leq n-1$, such that after a change of coordinates $\pol=\p\times\R^{n-k}$.
\item[(iii)] $\pol$ contains a line $\ell$.
\end{itemize}
\end{lem}

\section{Complement of basic nowhere dense semialgebraic sets}\label{s3}

The purpose of this section is to reduce the proof of Theorem \ref{mainext} to the case of an $n$-dimensional convex polyhedron. To that end, we prove in this section the following more general result. Given polynomials $g_1,\dots,g_r\in\R[\x]$ consider the \em basic semialgebraic set \em $\Ss:=\{g_1*_10,\dots,g_r*_r0\}\subset\R^n$, where each $*_i$ represents either the symbol $>$ or $\geq$. 

\begin{thm}\label{semial}
Let $\Ss\subsetneq\R^n$ be a proper basic semialgebraic set. Then  $\R^{n+1}\setminus(\Ss\times\{0\})$ is a polynomial image of $\R^{n+1}$.
\end{thm}
\begin{remark}\label{semialr}
In particular, if $\pol\subset\R^n$ is a convex polyhedron of dimension $d<n$ which is not a hyperplane, then $\R^n\setminus\pol$ and $\R^n\setminus\Int(\pol)$ are polynomial images of $\R^n$.
\end{remark}

\begin{lem}\label{presemial}
Let $\Lambda\subset\R^n$ be a (non necessarily semialgebraic) set and let $g\in\R[\x]$ be a polynomial. Then  there are polynomial maps $f_i:\R^{n+1}\to\R^{n+1}$ such that 
\begin{align*}
&f_1(\R^{n+1}\setminus(\Lambda\times\{0\}))=\R^{n+1}\setminus((\Lambda\cap\{g\geq0\})\times\{0\})\quad\text{and}\\ 
&f_2(\R^{n+1}\setminus(\Lambda\times\{0\}))=\R^{n+1}\setminus((\Lambda\cap\{g>0\})\times\{0\}).
\end{align*}
\end{lem}
\begin{proof} 
Denote $x:=(x_1,\dots,x_n)$ and consider the polynomial maps:
\begin{align*}
f_1:\R^{n+1}\to\R^{n+1},\ (x,x_{n+1})&\mapsto(x,x_{n+1}(1+x_{n+1}^2g(x))),\\[3pt]
f_2:\R^{n+1}\to\R^{n+1},\ (x,x_{n+1})&\mapsto(x,x_{n+1}(1+x^2_{n+1}g(x))(g^2(x)+(x_{n+1}-1)^2)).
\end{align*}
Let us see that those maps satisfy the desired conditions. Indeed, denote 
$$
\ell_a:=\{x_1=a_1,\dots,x_n=a_n\}\subset\R^{n+1}
$$ 
for each $a:=(a_1,\ldots,a_n)\in\R^n$, and observe that $f_1(\ell_a)=f_2(\ell_a)=\ell_a$. For each $a\in\R^n$, consider the univariate polynomials
$$
\begin{cases}
\alpha_a(\t):=\t(1+\t^2g(a)),\\
\beta_a(\t):=\t(1+\t^2g(a))(g^2(a)+(\t-1)^2),
\end{cases}
$$
which arise when we respectively substitute $x=a$ and $x_{n+1}=t$ in the last coordinates of $f_1$ and $f_2$. Observe that $\alpha_a$ has odd degree and
$$
f_1(\ell_a\setminus\{(a,0)\})=\{a\}\times\alpha_a(\R\setminus\{0\}))=
\begin{cases}
\ell_a\setminus\{0\}&\text{if $g(a)\geq0$,}\\
\ell_a&\text{if $g(a)<0$.}
\end{cases}
$$
This comes from the following facts
$$
\begin{cases}
\lim_{t\to\pm\infty}\alpha_a(t)=\pm\infty\ \text{and}\ \alpha_a^{-1}(0)=\{0\}&\text{if $g(a)\geq0$,}\\
\lim_{t\to\pm\infty}\alpha_a(t)=\mp\infty\ \text{and}\ \alpha_a^{-1}(0)=\{0,\pm\sqrt{-1/g(a)}\}&\text{if $g(a)<0$.}
\end{cases}
$$
Since $\R^{n+1}=\bigcup_{a\in\R^n}\ell_a$, one has
$$
\R^{n+1}\setminus\{\Lambda\times\{0\}\}=\bigcup_{a\in\R^n\setminus\Lambda}\ell_a\cup\bigcup_{a\in\Lambda\cap\{g\geq0\}}(\ell_a\setminus\{(a,0)\})\cup\bigcup_{a\in\Lambda\cap\{g<0\}}(\ell_a\setminus\{(a,0)\})
$$
and so
\begin{multline*}
f_1(\R^{n+1}\setminus(\Lambda\times\{0\}))=\bigcup_{a\in\R^n\setminus\Lambda}\ell_a\cup\bigcup_{a\in\Lambda\cap\{g\geq0\}}(\ell_a\setminus\{(a,0)\})\cup\bigcup_{a\in\Lambda\cap\{g<0\}}\ell_a\\
=((\R^n\setminus(\Lambda\cap\{g\geq0\}))\times\R)\cup((\Lambda\cap\{g\geq0\})\times(\R\setminus\{0\}))\\
=\R^{n+1}\setminus((\Lambda\cap\{g\geq0\})\times\{0\}).
\end{multline*}

The argument for the polynomial map $f_2$ is analogous. The polynomial $\beta_a$ has odd degree and
$$
f_2(\ell_a\setminus\{(a,0)\})=(a,\beta_a(\R\setminus\{0\}))=
\begin{cases}
\ell_a\setminus\{0\}&\text{if $g(a)>0$},\\
\ell_a&\text{if $g(a)=0$},\\
\ell_a&\text{if $g(a)<0$.}
\end{cases}
$$
This comes from the following facts
$$
\begin{cases}
\lim_{t\to\pm\infty}\beta_a(t)=\pm\infty\ \text{and}\ \beta_a^{-1}(0)=\{0\}&\text{if $g(a)>0$,}\\
\lim_{t\to\pm\infty}\beta_a(t)=\pm\infty\ \text{and}\ \beta_a^{-1}(0)=\{0,1\}&\text{if $g(a)=0$,}\\
\lim_{t\to\pm\infty}\beta_a(t)=\mp\infty\ \text{and}\ \beta_a^{-1}(0)=\{0,\pm\sqrt{-1/g(a)}\}&\text{if $g(a)<0$.}
\end{cases}
$$
Thus, since
$$
\R^{n+1}\setminus\{\Lambda\times\{0\}\}=\bigcup_{a\in\R^n\setminus\Lambda}\ell_a\cup\bigcup_{a\in\Lambda\cap\{g\leq0\}}(\ell_a\setminus\{(a,0)\})\cup\bigcup_{a\in\Lambda\cap\{g>0\}}(\ell_a\setminus\{(a,0)\}),
$$
we conclude that
\begin{multline*}
f_2(\R^{n+1}\setminus(\Lambda\times\{0\}))=\bigcup_{a\in\R^n\setminus\Lambda}\ell_a\cup\bigcup_{a\in\Lambda\cap\{g\leq0\}}\ell_a\cup\bigcup_{a\in\Lambda\cap\{g>0\}}(\ell_a\setminus\{(a,0)\})\\
=((\R^n\setminus(\Lambda\cap\{g\leq0\}))\times\R)\cup((\Lambda\cap\{g>0\})\times(\R\setminus\{0\}))\\
=\R^{n+1}\setminus((\Lambda\cap\{g>0\})\times\{0\}),
\end{multline*}
as wanted.
\end{proof}

Now, we are ready to prove Theorem \ref{semial}.

\begin{proof}[Proof of Theorem \em \ref{semial}] 
Let $\Ss:=\{g_1*_10,\dots,g_r*_r0\}$ be a proper basic semialgebraic subset of $\R^n$, where each $*_i\in\{>,\ \ge\}$. We proceed by induction on the number $r$ of inequalities needed to describe $\Ss$. Since $\Ss\subsetneq\R^n$ we may assume that ${\bf 0}\notin\Ss$. If $r=1$, we have $\Ss:=\{g_1*_10\}$ and we consider the polynomial map
$$
f_0:\R^{n+1}\to\R^{n+1},\quad (x_1,\dots,x_{n+1})\mapsto(x_1x_{n+1},\dots,x_nx_{n+1},x_{n+1}),
$$
whose image is $\R^{n+1}\setminus(\Lambda\times\{0\})$, where $\Lambda:=\R^n\setminus\{{\bf 0}\}$. We apply Lemma \ref{presemial} to $g:=g_1$ and we choose 
$$
f:=\begin{cases}
f_1&\text{if $\{g*_10\}=\{g\geq0\}$},\\
f_2&\text{if $\{g*_10\}=\{g>0\}$}.
\end{cases}
$$
Since ${\bf 0}\notin\Ss$ we have $\Lambda\cap\Ss=\Ss$; hence, the composition $f\circ h$ satisfies the equality
$$
(f\circ f_0)(\R^{n+1})=f(\R^{n+1}\setminus(\Lambda\times\{0\}))=\R^{n+1}\setminus((\Lambda\cap\Ss)\times\{0\})=\R^{n+1}\setminus(\Ss\times\{{0}\}),
$$
as wanted.

Suppose now that the result holds for each proper basic semialgebraic set of $\R^n$ described by $r-1\geq1$ inequalities, and let $\Ss:=\{g_1*_10,\dots, g_r*_r0\}$ be a proper basic semialgebraic subset of $\R^n$. We choose $\Lambda:=\{g_1*_10,\dots,g_{r-1}*_{r-1}0\}$ (that we may assume to be proper because otherwise $\Ss:=\{g_r*_r0\}$ can be described using just one inequality and this case has already been studied) and, by the inductive hypothesis, there is a polynomial map $f_0:\R^{n+1}\to\R^{n+1}$ such that $f_0(\R^{n+1})=\R^{n+1}\setminus(\Lambda\times\{0\})$. By Lemma \ref{presemial} for $g:=g_r$, and choosing 
$$
f:=\begin{cases}
f_1&\text{if $\{g*_r0\}=\{g\geq0\}$},\\
f_2&\text{if $\{g*_r0\}=\{g>0\}$},
\end{cases}
$$
there exists a polynomial map $f:\R^{n+1}\to\R^{n+1}$ such that
$$
f(\R^{n+1}\setminus(\Lambda\times\{0\}))=\R^{n+1}\setminus((\Lambda\cap\{g_r*_r0\})\times\{0\})=\R^{n+1}\setminus(\Ss\times\{0\}).
$$
Finally, the composition $f\circ f_0$ provides us
$$
(f\circ f_0)(\R^{n+1})=f(\R^{n+1}\setminus(\Lambda\times\{0\}))=\R^{n+1}\setminus(\Ss\times\{0\}),
$$
as wanted.
\end{proof}

\section{Trimming tools of type one}\label{s4}

The purpose of this section is to present and develop the main tools that we will use in Section~\ref{s5} to prove the part of Theorem \ref{mainext} concerning complements of $n$-dimensional convex polyhedra.

\subsection{First trimming position and trimming maps of first type}\label{lonchas}
Consider in $\R^n$ the natural fibration arising from the projection
$$
\pi_{n}:\R^n\to\R^n,\ (x_1,\dots,x_n)\mapsto (x_1,\dots,x_{n-1},0).
$$
The fiber $\pi_n^{-1}(a,0)$ of $\pi_n$ is a line parallel to the vector $\vec{e}_n:=(0,\ldots,0,1)$ for each $a\in\R^{n-1}$. Given an $n$-dimensional convex polyhedron $\pol\subset\R^n$, the intersection 
\begin{equation}\label{iia}
\Ii_a:=\pi_n^{-1}(a,0)\cap\pol
\end{equation} 
may be either empty or a closed interval (bounded or not). Denote by ${\mathfrak F}\subset\R^{n-1}$ the set of all points $a\in\R^{n-1}$ such that $\Int(\pol)\cap\pi_n^{-1}(a,0)=\varnothing$ and observe that the points $a\in{\mathfrak F}$ satisfy either $\Ii_a=\varnothing$ or $\Ii_a\subset\partial\pol$; in the latter case, $\Ii_a$ is contained in a facet of $\pol$. It holds that
$$
\Int(\pol)\cap\pi_n^{-1}(a,0)=\left\{\begin{array}{cl}
\varnothing&\text{if $a\in{\mathfrak F}$,}\\
\Int\Ii_a&\text{otherwise;}
\end{array}\right.
$$
hence, $\Int(\pol)=\bigsqcup_{a\in\R^{n-1}\setminus{\mathfrak F}}\Int\Ii_a$. 

\begin{lem}\label{prelonchas1} 
Let $\pol\subset\R^n$ be an $n$-dimensional convex polyhedron, let $a\in\R^{n-1}$ be such that $\Ii_a\neq\varnothing$, and take $p\in\partial\,\Ii_a$. Then  there is a facet $\Ff$ of $\pol$ nonparallel to $\vec{e}_n$ such that $p\in\Ff$.
\end{lem}
\begin{proof} 
It is enough to prove that there is a facet $\Ff$ of $\pol$ passing through $p$, which does not contain $\Ii_a$. To that end, we proceed by induction on the dimension $n$ of $\pol$. If $n=1$, $\Ii_a=\pol$ and obviously $p$ belongs to a facet of $\pol$ that does not contain $\Ii_a$. For $n>1$, we pick a facet $\Ff_0$ of $\pol$ passing through $p$. If $\Ff_0$ does not contain $\Ii_a$, we are done; otherwise, $\Ff_0$ is an $(n-1)$-dimensional convex polyhedron such that $\Ii_a=\pi_{n}^{-1}(a,0)\cap\Ff_0$. Thus, by induction hypothesis, there is a facet $\Ee$ of $\Ff_0$ passing through $p$, but which does not contain $\Ii_a$. Since $\dim\Ee=n-2$ and $\dim\Ff_0=n-1$, there is a facet $\Ff$ of $\pol$ such that $\Ee=\Ff_0\cap\Ff$; since $\Ii_a\subset\Ff_0$ but $\Ii_a\not\subset\Ee$, we conclude that $\Ii_a\not\subset\Ff$, as wanted.
\end{proof}

\begin{defines}\label{def5a} 
(1) An $n$-dimensional convex polyhedron $\pol\subset\R^n$ is in \emph{weak first trimming position} if:
\begin{itemize}
\item[(i)] For each $a\in\R^{n-1}$ the set $\Ii_a$ is bounded. 
\end{itemize}

If $\pol$ satisifies also that:
\begin{itemize}
\item[(ii)] The set ${\mathfrak A}_\pol:=\{a\in\R^{n-1}:\ a_{n-1}\leq0, \Ii_a\neq\varnothing,\ (a,0)\not\in\Ii_a\}$ is bounded,
\end{itemize}
we will say that $\pol$ is in \emph{strong first trimming position}.

Besides, if ${\mathfrak A}_\pol=\varnothing$ we say that the strong first trimming position of $\pol$ is \em optimal\em.

(2) On the other hand, we say that the (weak or strong) first trimming position of $\pol\subset\R^n$ is \em extreme with respect to the facet $\Ff$ \em if $\Ff\subset\{x_{n-1}=0\}$ and $\pol\subset\{x_{n-1}\leq0\}$.
\end{defines}

\begin{figure}[t]
\begin{minipage}[b]{0.265\linewidth}
\begin{center}
\begin{tikzpicture}[line join=round,background rectangle/.style=
{double,draw=gray},
show background rectangle,inner frame sep=0.2cm, scale=0.4]
\filldraw[fill=gray!80,fill opacity=0.8,draw=gray,style=dashed](-4,-2)--(0,1.5)--(0,3)--(-4,4)--cycle;
\draw[ultra thick,<->] (-4,-2)--(0,1.5)--(0,3)--(-4,4);
\draw (-3.5,1.5) node{$\EuScript \pol$};

\draw[<->] (-5,0) -- (3,0)node[below] {$x_1$};
\draw[<->] (0,-3) -- (0,4)node[above] {$x_2$};
\draw[densely dashed,very thin] (-1,3.25)--(0,3.25) node[right]{$q$};
\draw[densely dashed,very thin] (-1,0.75)--(0,0.75);
\draw (0,0.75)node[right]{$p$};
\draw[very thick, gray] (-1,0.75)--(-1,3.25);
\filldraw[fill=gray,draw=gray] (-1,0.75) circle (0.5mm);
\filldraw[fill=gray,draw=gray] (-1,3.25) circle (0.5mm);
\path (-1,2) node[left]{$\EuScript I_a$} (-1,0) node[below]{$a$};

\end{tikzpicture}\\
(a)\vspace{4pt}\\
\begin{tikzpicture}[line join=round,background rectangle/.style=
{double,draw=gray},
show background rectangle,inner frame sep=0.2cm,scale=0.4]

\filldraw[fill=gray!80,fill opacity=0.8,draw=gray,style=dashed](-4,-2)--(0,-1)--(0,3)--(-4,4)--cycle;
\draw[ultra thick,<->] (-4,-2)--(0,-1)--(0,3)--(-4,4);
\draw (-3.5,1.5) node{$\EuScript \pol$};

\draw[<->] (-5,0) -- (3,0)node[below] {$x_1$};
\draw[<->] (0,-3) -- (0,4)node[above] {$x_2$};
\draw[densely dashed,very thin] (-1,3.25)--(0,3.25) node[right]{$q$};
\draw[densely dashed,very thin] (-1,-1.25)--(0,-1.25) node[right]{$p$};
\draw[very thick, gray] (-1,-1.25)--(-1,3.25);
\filldraw[fill=gray,draw=gray] (-1,-1.25) circle (0.5mm);
\filldraw[fill=gray,draw=gray] (-1,3.25) circle (0.5mm);
\path (-1,2) node[left]{$\EuScript I_a$};

\end{tikzpicture}\\
(b)
\end{center}
\label{fig:figure1ab}
\end{minipage}
\hspace{0.5cm}
\begin{minipage}[b]{0.55\linewidth}
\begin{center}
\begin{tikzpicture}[line join=round,background rectangle/.style=
{double,draw=gray},
show background rectangle,inner frame sep=0.3cm, xscale=0.64, yscale=0.671]
\tikzstyle{planelinestyle} = [cull=false,red,dashed]
\tikzstyle{blinestyle} = [cull=false,red,dashed]
\tikzstyle{plinestyle} = [cull=false,red,dashed]
\tikzstyle{verstyle}=[cull=false,fill=blue!20,fill opacity=0.8]
\tikzstyle{horstyle}=[cull=false,fill=blue!20,fill opacity=0.8]
\tikzstyle{boxstyle}=[cull=false,fill=blue!20,fill opacity=0.8]
\tikzstyle{polystyle} = [cull=false,fill=blue!20,fill opacity=0.8]
\filldraw[fill=gray!40,fill opacity=0.3,draw=none](-4.529,1.604)--(-6.293,.091)--(-6.293,-5.79)--(-4.529,-.16)--cycle;
\draw[dashed,ultra thick,draw=white!60](-1.588,-4.891)--(-1.588,-.186);
\draw[draw=black!60,dashed,ultra thin](-6.293,.091)--(-6.293,-5.79)--(-4.529,-.16)--(-4.529,1.604);
\filldraw[fill=black!60,fill opacity=0.8,draw=none](-1,.515)--(-4.529,-.16)--(-4.529,1.604)--cycle;
\draw[draw=black,very thick](-1,.515)--(-4.529,-.16)--(-4.529,1.604);
\filldraw[fill=black!60,fill opacity=0.8,draw=none](-6.293,-5.79)--(-4.529,-.16)--(-1,.515)--(-.076,-.065)--(-1.588,-4.891)--cycle;
\draw[draw=black,very thick](-.076,-.065)--(-1.588,-4.891)--(-6.293,-5.79)--(-4.529,-.16)--(-1,.515);
\draw[dashed,ultra thick,draw=white!60](-4.529,1.604)--(-4.529,-.16)--(-6.293,-5.79)--(-1.588,-4.891)--(-1.588,-1.362);
\filldraw[fill=gray!40,fill opacity=0.3,draw=none](-6.293,.091)--(-1.588,-1.362)--(-1.588,-4.891)--(-6.293,-5.79)--cycle;
\draw[draw=black!60,dashed,ultra thin](-1.588,-1.362)--(-1.588,-4.891)--(-6.293,-5.79)--(-6.293,.091);
\filldraw[fill=green!40,fill opacity=0.4,draw=none](-1.588,-1.362)--(-1.588,-6.067)--(1.588,-3.343)--(1.588,1.362)--cycle;
\draw[very thin](-1.588,-1.362)--(-1.588,-6.067)--(1.588,-3.343)--(1.588,1.362);
\filldraw[fill=black!60,fill opacity=0.8,draw=none](-1.588,-1.362)--(-.076,-.065)--(-1.588,-4.891)--cycle;
\draw[draw=black,very thick](-.076,-.065)--(-1.588,-4.891);
\draw[thick,arrows=<->](-2.117,-1.816)--(0,0)--(3.529,-1.09);
\filldraw[fill=blue!40,fill opacity=0.8,very thin](-3.117,2.815)--(4.235,.545)--(1.059,-2.179)--(-6.293,.091)--cycle;
\filldraw[fill=gray!40,fill opacity=0.3,draw=none](-4.529,1.604)--(-4.529,3.369)--(-6.293,1.855)--(-6.293,.091)--cycle;
\draw[draw=black!60,dashed,ultra thin](-4.529,1.604)--(-4.529,3.369)--(-6.293,1.855)--(-6.293,.091);
\filldraw[fill=black!60,fill opacity=0.8,draw=none](-1,.515)--(-4.529,1.604)--(-4.529,3.369)--(.176,1.328)--(.176,.739)--cycle;
\draw[draw=black,very thick](-4.529,1.604)--(-4.529,3.369)--(.176,1.328)--(.176,.739)--(-1,.515);
\filldraw[fill=black!60,fill opacity=0.8,draw=none](-1,.515)--(.176,.739)--(-.076,-.065)--cycle;
\draw[draw=black,very thick](-1,.515)--(.176,.739)--(-.076,-.065);
\draw[dashed,ultra thick,draw=white!60](-6.293,1.855)--(-4.529,3.369)--(-4.529,1.604);
\filldraw[draw=black,very thick,fill=black!60,fill opacity=0.8](-6.293,1.855)--(-4.529,3.369)--(.176,1.328)--(-1.588,-.186)--cycle;
\filldraw[fill=gray!40,fill opacity=0.3,draw=none](-6.293,.091)--(-6.293,1.855)--(-1.588,-.186)--(-1.588,-1.362)--cycle;
\draw[draw=black!60,dashed,ultra thin](-6.293,.091)--(-6.293,1.855)--(-1.588,-.186)--(-1.588,-1.362);
\draw[dashed,ultra thick,draw=white!60](-1.588,-.186)--(-6.293,1.855);
\filldraw[fill=green!40,fill opacity=0.4,draw=none](-1.588,1.578)--(-1.588,-1.362)--(1.588,1.362)--(1.588,4.303)--cycle;
\draw[very thin](1.588,1.362)--(1.588,4.303)--(-1.588,1.578)--(-1.588,-1.362);
\draw[dashed,ultra thick,draw=white!60](-1.588,-1.362)--(-1.588,-.186);
\draw[thick,arrows=->](0,0)--(0,3.529);
\filldraw[fill=black!60,fill opacity=0.8,draw=none](-1.588,-1.362)--(-1.588,-.186)--(.176,1.328)--(.176,.739)--(-.076,-.065)--cycle;
\draw[draw=black,very thick](-1.588,-.186)--(.176,1.328)--(.176,.739)--(-.076,-.065);
\path (-2.117,-1.816) node[below]{$x_1$} (3.529,-1.09) node[below]{$x_2$} (0,3.529) node[above]{$x_3$};
\path (-5.999,-1.47) node{$\EuScript K$};
\path (1.588,4.303) node[above]{$x_2=0$};
\path (3.2,1.5) node[below right]{$x_3=0$};
\end{tikzpicture}\\
(c)
\end{center}
\label{fig:figure1c}
\end{minipage}
\vspace{-5mm}
\caption{{\small (a) shows a convex polygon in strong first trimming position, (b) shows a convex polygon in optimal strong first trimming position, and (c) shows a convex polyhedron in strong first trimming position.}}\label{fig:1}
\end{figure}
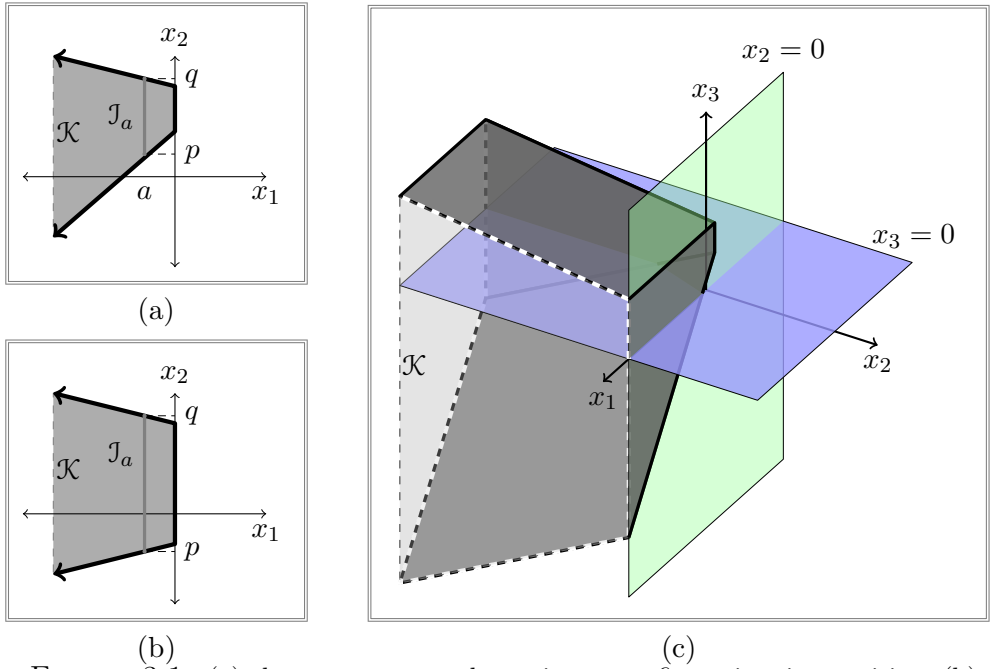

\begin{lem}\label{aajuste}\em
Let $\pol\subset\R^n$ be an $n$-dimensional convex polyhedron in strong first trimming position. Then  there is $N>0$ such that
$$
A:=\bigcup_{a\in{\mathfrak A}_\pol}\partial\Ii_a\subset{\mathfrak A}_\pol\times[-N,N]
$$
\end{lem}
\begin{proof}
Indeed, since $\pol$ is in strong first trimming position, the set $\Ii_a$ is bounded for each $a\in\R^{n-1}$. By Lemma \ref{prelonchas1}, for each $a\in{\mathfrak A}_\pol$ the extrems $(a,p_a)$ and $(a,q_a)$ of the interval $\Ii_a=\{a\}\times[p_a,q_a]$ belong to facets of $\pol$ that are not parallel to the vector $\vec{e}_n$. Let $\Ff_1,\ldots,\Ff_s$ be the facets of $\pol$ nonparallel to the vector $\vec{e}_n$ and $h_i:=b_{0i}+b_{1i}\x_1+\cdots+b_{ni}\x_n$ a degree one equation of the hyperplane $H_i$ generated by $\Ff_i$ for $1\leq i\leq s$. Since $\Ff_i$ is nonparallel to the vector $\vec{e}_n$ we may assume that $b_{ni}=1$. Define $g_i:=-b_{0i}-b_{1i}\x_1-\cdots-b_{n-1,i}\x_{n-1}$ and observe that for each $a\in{\mathfrak A}_\pol$ there are $1\leq i,j\leq s$ such that $g_i(a)=p_a$ and $g_{j}(a)=q_a$; hence, $p_a,q_a\in\{g_1(a),\ldots,g_s(a)\}$. Since ${\mathfrak A}_\pol$ is bounded, its closure $\ol{\mathfrak A}_\pol$ is compact, and so the continuous functions $g_i$ are bounded over $\ol{\mathfrak A}_\pol$; hence, there is $N>0$ such that $p_a,q_a\in[-N,N]$ for each $a\in{\mathfrak A}_\pol$, as wanted.
\end{proof}

\begin{define}\label{1tm}
Let $P,Q\in\R[\x]$ be two nonzero polynomials such that $\deg_{\x_n}(Q(a,\x_n))\leq2\deg_{\x_n}(P(a,\x_n))$ for each $a\in\R^{n-1}$ and $Q$ is strictly positive on $\R^n$. Denote $\beta_{P,Q}$ the regular function
$$
\beta_{P,Q}(\x):=\x_n\Big(1-\x_{n-1}\frac{P^2(\x)}{Q(\x)}\Big),
$$
and define the \emph{trimming map $\TF_{P,Q}$ of type} I associated to $P$ and $Q$ as the regular map
$$
\TF_{P,Q}:\R^n\to\R^n,\ x:=(x_1,\dots,x_n)\mapsto(x_1,\dots,x_{n-1},\beta_{P,Q}(x)).
$$ 
Observe that if $Q:=M>0$ is constant then $\TF_{P,M}$ is a polynomial map. 

Recall that the \em degree \em of a rational function $h:=f/g\in\R(\x):=\R(\x_1,\ldots,\x_n)$ is the difference between the degrees of the numerator $f$ and the denominator $g$.
\end{define}
\begin{remarks}\label{veryeasy}
The following properties are straightforward:
\begin{itemize}
\item[(i)] For each $a\in\R^{n-1}$, the regular function $\beta_{P,Q}^a(\x_n):=\beta_{P,Q}(a,\x_n)\in\R(\x_n)$ (which depends just on the variable $\x_n$) has odd degree $\geq1$; hence, $\beta^a_{P,Q}(\R)=\R$ for each $a\in\R^{n-1}$. Moreover, if $a':=(a_1,\ldots,a_{n-2})\in\R^{n-2}$ we have $\beta_{P,Q}(a',0,\x_n)=\x_n$.
\item[(ii)] For each $a\in\R^{n-1}$ the equality $\TF_{P,Q}(\pi_n^{-1}(a,0))=\pi_n^{-1}(a,0)$ holds, and so $\TF_{P,Q}(\R^n)=\R^n$.
\item[(iii)] The set of fixed points of $\TF_{P,Q}$ is $\{x_{n-1}x_nP(x)=0\}$.
\item[(iv)] The condition $\deg_{\x_n}(Q(a,\x_n))\leq2\deg_{\x_n}(P(a,\x_n))$ for each $a\in\R^{n-1}$ holds if $\deg_{\x_n}(Q)\leq 2\deg(P)=2\deg_{\x_n}(P)$.
\item[(iv)] Let $\pol\subset\R^n$ be an $n$-dimensional convex polyhedron and let $\Ff_1,\ldots,\Ff_r$ be the facets of $\pol$ which are contained in hyperplanes nonparallel to $\vec{e}_n$. Let $f_1,\dots,f_r\in\R[\x]$ be degree one polynomial equations of the hyperplanes $H_1,\ldots,H_r$ respectively generated by the facets $\Ff_1,\ldots,\Ff_r$, and define $P:=f_1\cdots f_r$. Then  $P$ is identically cero on the facets $\Ff_1,\ldots,\Ff_r$ and has degree $\deg(P)=\deg_{\x_n}(P)=r$.
\end{itemize}
\end{remarks}

Next, we recall the following elementary polynomial bounding inequality.

\begin{lem}\label{acotacion}
Let $P\in\R[\x]$ be a polynomial of degree $d\geq0$, let $M$ be the sum of the absolute values of the nonzero coefficients of $P$ and let $m$ be a positive integer such that $d\leq2m$. Then  $|P(x)|\leq M(1+\|x\|^2)^m$ for each $x\in\R^n$.
\end{lem}
\begin{proof}
Write $P:=\sum_{a_\nu\neq0}a_\nu\x^\nu$, where $\nu:=(\nu_1,\ldots,\nu_n)$, $\x:=(\x_1,\ldots,\x_n)$ and $\x^\nu:=\x_1^{\nu_1}\cdots\x_n^{\nu_n}$. For each $x:=(x_1,\ldots,x_n)\in\R^n$ we have
$$
|x^\nu|=|x_1|^{\nu_1}\cdots|x_n|^{\nu_n}\leq\prod_{i=1}^n(1+\|x\|^2)^{\nu_i/2}\leq(1+\|x\|^2)^m;
$$
hence,
$$
|P(x)|\leq\sum_{a_\nu\neq0}|a_\nu||x^\nu|\leq\sum_{a_\nu\neq0}|a_\nu|(1+\|x\|^2)^m=M(1+\|x\|^2)^m,
$$
as wanted. 
\end{proof}

\begin{lem}\label{bp} 
Let $P\in\R[\x]$ be a nonzero polynomial such that $d:=\deg(P)=\deg_{\x_n}(P)$. Then 
\begin{itemize}
\item[(i)] For each $Q\in\R[\x]$ strictly positive with $\deg_{\x_n}(Q)\leq 2\deg(P)$, the partial derivative $\frac{\partial\beta_{P,Q}}{\partial\x_n}$ has constant value $1$ on the set $\{x_{n-1}P(x)=0\}$.
\item[(ii)] For each compact $K\subset\R^n$ there is a positive integer $M_K$ such that if $M\geq M_K$ the partial derivative $\frac{\partial\beta_{P,M}}{\partial\x_n}$ is strictly positive over the set $K\cap\{x_{n-1}\leq0\}$.
\item[(iii)] The strictly positive polynomials $Q_M:=M(1+\x_{n-1}^2)(1+\|\x\|^2)^d$ have $\deg_{\x_n}(Q_M)=2\deg(P)=2\deg_{\x_n}(P)$ and there is $M_0>0$ such that the partial derivative $\frac{\partial\beta_{P,Q_M}}{\partial\x_n}$ is strictly positive on the set $\{x_{n-1}\leq0\}$ for each $M\geq M_0$.
\end{itemize}
\end{lem}
\begin{proof}
Observe first that
$$
\frac{\partial\beta_{P,Q}}{\partial\x_n}=\Big(1-\x_{n-1}\frac{P^2(\x)}{Q(\x)}\Big)-\x_{n-1}\x_n\frac{P(\x)}{Q(\x)}\Big(2\frac{\partial P}{\partial\x_n}(\x)-\frac{\partial Q}{\partial\x_n}(\x)\frac{P(\x)}{Q(\x)}\Big)
$$
so that (i) clearly holds.

To prove (ii), denote by $M_K\geq0$ the maximum of the continuous function 
$$
\R^n\to\R,\ x:=(x_1,\dots,x_n)\mapsto\Big|2x_{n-1}x_n P(x)\frac{\partial P}{\partial\x_n}(x)\Big|.
$$ 
over the compact $K\cap\{x_{n-1}\leq0\}$. If $M\geq M_K$, then $\frac{\partial\beta_{P,M}}{\partial\x_n}$ is strictly positive on $K\cap\{x_{n-1}\leq0\}$.

To prove (iii), observe first that $\deg(\x_n P\frac{\partial P}{\partial\x_n})\leq 2d$ and so, by Lemma \ref{acotacion}, there is $M_0>0$ such that for each $M\geq M_0$
$$
\Big|x_n P(x)\frac{\partial P}{\partial\x_n}(x)\Big|+|d P^2(x)|\leq \frac{M}{2}(1+\|x\|^2)^d
$$
for each $x\in\R^n$. Fix $M\geq M_0$ and, for simplicity, denote $Q:=Q_M$; observe that
$$
\frac{\partial Q}{\partial\x_n}(\x)\frac{1}{Q(\x)}=\frac{2d\x_n}{1+\|\x\|^2}.
$$
Using the facts that $2|x_{n-1}|\leq(1+x_{n-1}^2)$ and $|x_n|^2<1+\|x\|^2$ for each $x\in\R^n$, we deduce that
\begin{equation*}
\begin{split}
\frac{\partial\beta_{P,Q}}{\partial\x_n}(x)&\geq\Big(1-x_{n-1}\frac{P^2(x)}{Q(x)}\Big)-\frac{2|x_{n-1}|}{Q}\Big(\Big|x_nP(x)\frac{\partial P}{\partial\x_n}(x)\Big|+|dP^2(x)|\frac{|x_n|^2}{1+\|x\|^2}\Big)\\
&\geq\frac{1}{2}-x_{n-1}\frac{P^2(x)}{Q(x)}>0,
\end{split}
\end{equation*}
for each $x\in\{x_{n-1}\leq0\}$.
\end{proof}

\begin{lem}\label{trim}
Let $\pol\subset\R^n$ be an $n$-dimensional convex polyhedron in weak first trimming position and let $\Ff_1,\ldots,\Ff_r$ be the facets of $\pol$ contained in hyperplanes nonparallel to the vector $\vec{e}_n$. Let $P\in\R[\x]$ be a nonzero polynomial which is identically zero on $\Ff_1,\ldots,\Ff_r$ and has degree $\deg(P)=\deg_{\x_n}(P)$ \em (see Remark \ref{veryeasy})\em. Then  there is a strictly positive $Q\in\R[\x]$ with $\deg_{\x_n}(Q)\leq2\deg(P)$ such that 
\begin{itemize}
\item[(i)] $\mathsf{F}_{P,Q}(\R^n\setminus\pol)=\R^n\setminus(\pol\cap\{x_{n-1}\leq0\})$, 
\item[(ii)] $\mathsf{F}_{P,Q}(\R^n\setminus\Int(\pol))=\R^n\setminus(\Int(\pol)\cap\{x_{n-1}\leq0\})$.
\end{itemize}
Moreover, if $\pol$ is in strong first trimming position, we can choose $Q$ to be constant.
\end{lem}
\begin{proof}
Denote $\pol':=\pol\cap\{x_{n-1}\leq0\}$ and let ${\mathfrak F}\subset\R^{n-1}$ be the set of those $a\in\R^{n-1}$ such that $\pi_n^{-1}(a,0)\cap\Int(\pol)=\varnothing$. For each $a\in\R^{n-1}$ consider the following sets:
$$
\begin{array}{ll}
\Mm_a:=\{t\in\R:(a,t)\in\R^n\setminus\pol\}\ &\ \Nn_a:=\{t\in\R:(a,t)\in\R^n\setminus\Int(\pol)\}\\
\hspace{1.5mm}\Ss_a:=\{t\in\R:(a,t)\in\R^n\setminus\pol'\}\ &\ \Tt_a:=\{t\in\R:(a,t)\in\R^n\setminus(\Int(\pol)\cap\{x_{n-1}\le 0\})\}.
\end{array}
$$
Observe that $\Mm_a\subset\Nn_a$, $\Ss_a\subset\Tt_a$, and 
$$
\Nn_a=\left\{\begin{array}{cl}
\R&\ \text{ if $a\in{\mathfrak F}$,}\\[4pt]
\ol{\Mm}_a&\ \text{ otherwise.}
\end{array}\right.
$$
$$
\Ss_a=\left\{\begin{array}{ll}
\Mm_a&\ \text{ if $a_{n-1}\leq0$,}\\[4pt]
\R&\ \text{ if $a_{n-1}>0$,}
\end{array}\right.
\qquad\text{and}\quad
\Tt_a=\left\{\begin{array}{ll}
\Nn_a&\ \text{ if $a_{n-1}\leq0$,}\\[4pt]
\R&\ \text{ if $a_{n-1}>0$.}
\end{array}\right.
$$
By \S\ref{lonchas}, we have: 
\begin{itemize}
\item[(1)] $\R^n\setminus\pol=\bigsqcup_{a\in\R^{n-1}}(\{a\}\times\Mm_a)$ and $\R^n\setminus\pol'=\bigsqcup_{a\in\R^{n-1}}(\{a\}\times\Ss_a)$.
\item[(2)] $\R^n\setminus\Int(\pol)=\bigsqcup_{a\in\R^{n-1}}(\{a\}\times\Nn_a)$. 
\item[(3)] $\R^n\setminus(\Int(\pol)\cap\{x_{n-1}\leq0\})=\bigsqcup_{a\in\R^{n-1}}(\{a\}\times\Tt_a)$.
\end{itemize}
 
\begin{substeps}{trim}\label{signder}
At this point, we choose $Q$ as follows. If $\pol$ is in strong first trimming position, ${\mathfrak A}_\pol$ is bounded (see Definition \ref{def5a}) and there is, by Lemma \ref{aajuste}, an $N>0$ such that $A:=\bigcup_{a\in{\mathfrak A}_\pol}\partial\Ii_a\subset{\mathfrak A}_\pol\times[-N,N]$, where $\Ii_a:=\pi_n^{-1}(a,0)\cap\pol$. By Lemma \ref{bp} (ii) applied to the compact set $K:=\ol{\mathfrak A}_\pol\times[-N,N]$, there is $Q:=M>0$ such that the partial derivative $\frac{\partial\beta_{P,M}}{\partial\x_n}$ is strictly positive on $K\cap\{x_{n-1}\leq0\}$. On the other hand, if $\pol$ is in weak first trimming position there exists, by Lemma \ref{bp}, a strictly positive polynomial $Q$ such that $\deg_{\x_n}(Q)=2\deg(P)=2\deg_{\x_n}(P)$ and the partial derivative $\frac{\partial\beta_{P,Q}}{\partial\x_n}$ is strictly positive on the set $\{x_{n-1}\leq0\}$.
\end{substeps}

Since the map $\mathsf{F}_{P,Q}$ preserves the lines $\pi_n^{-1}(a,0)$, to show equalities (i) and (ii) in the statement, it is enough to check that
\begin{equation}\label{recort}
\beta_{P,Q}^a(\Mm_a)=\Ss_a\quad\text{and}\quad\beta_{P,Q}^a(\Nn_a)=\Tt_a\qquad\forall\,a\in\R^{n-1}. 
\end{equation}

To check the identities \eqref{recort} we distinguish several cases attending to the particularities of a fixed $a\in\R^{n-1}$.

\vspace{2mm} 
\noindent{\bf Case 1}: If $\Mm_a=\R$, since $\beta_{P,Q}^a(\R)=\R$ (see Remark \ref{veryeasy}(i)) we have 
$$
\beta_{P,Q}^a(\Nn_a)=\beta_{P,Q}^a(\Mm_a)=\R=\Ss_a=\Tt_a.
$$ 

\vspace{2mm} 
\noindent{\bf Case 2}: If $a_{n-1}\le 0$ and $\Mm_a:={]{-\infty}, p[}\cup{]q,+\infty[}$, where $p\le q$, there are indices $1\leq i,j\leq r$ such that $(a,p)\in\Ff_i$, $(a,q)\in\Ff_j$ (see Lemma \ref{prelonchas1}) and either $\Nn_a={]{-\infty}, p]}\cup{[q,+\infty[}$ or $\Nn_a=\R$. In this last case $\beta_{P,Q}^a(\Nn_a)=\beta_{P,Q}^a(\R)=\R=\Tt_a$. 

Next, we compute the images $\beta^a_{P,Q}(\Mm_a)$ and $\beta^a_{P,Q}(\Nn_a)$ attending to the different possible situations.

\noindent (2.a) If $q\ge 0$, then $\beta_{P,Q}^a(q)=q$,\ $\lim_{t\to+\infty}\beta_{P,Q}^a(t)=+\infty$ and, since $a_{n-1}\leq0$,
$$
\beta_{P,Q}^a(t)=t\Big(1-\frac{a_{n-1}P^2(a,t)}{Q(a,t)}\Big)\ge t>q\quad\forall\,t>q.
$$
Now, by continuity $\beta_{P,Q}^a(]q,+\infty[)={]q,+\infty[}$ and $\beta_{P,Q}^a([q,+\infty[)=[q,+\infty[$. 

\noindent (2.b) Analogously, if $p\le0$ we have 
$$
\beta_{P,Q}^a(p)=p,\quad\lim_{t\to{-\infty}}\beta_{P,Q}^a(t)={-\infty}\quad\text{and}\quad\beta_{P,Q}^a(t)\leq t<p\quad\forall\,t<p.
$$ 
Again, by continuity, $\beta_{P,Q}^a(]{-\infty}, p[)={]{-\infty},p[}$ and $\beta_{P,Q}^a(]{-\infty}, p])={]{-\infty}, p]}$.
 
\noindent (2.c) Next, if $q<0$, it holds that $(a,0)\not\in\Ii_a$, that is, $a\in{\mathfrak A}_\pol$ and by \ref{trim}.\ref{signder} the derivative $\frac{\partial\beta_{P,Q}^a}{\partial\x_n}$ is strictly positive on the interval $[q,0]\subset[-N,N]$. In particular, $\beta_{P,Q}^a$ is an increasing function on the interval $[q,0]$; hence, $\beta_{P,Q}^a(]q,0])={]q,0]}$, because $\beta_{P,Q}^a(q)=q$ and $\beta_{P,Q}^a(0)=0$. Moreover, by (2.a), $\beta_{P,Q}^a(]0,+\infty[)={]0,+\infty[}$ and so $\beta_{P,Q}^a(]q,+\infty[)={]q,+\infty[}$ and $\beta_{P,Q}^a([q,+\infty[)=[q,+\infty[$.

\noindent (2.d) Analogously, if $p>0$, it holds that $(a,0)\not\in\Ii_a$, that is, $a\in{\mathfrak A}_\pol$ and by \ref{trim}.\ref{signder} the derivative $\frac{\partial\beta_{P,Q}^a}{\partial\x_n}$ is strictly positive in the interval $[0,p]\subset[-N,N]$. In particular, $\beta_{P,Q}^a$ is an increasing function on the interval $[0,p]$, and since $\beta_{P,Q}^a(p)=p$ and $\beta_{P,Q}^a(0)=0$, we deduce that $\beta_{P,Q}^a([0,p[)=[0,p[$. Moreover, by (2.b), $\beta_{P,Q}^a(]{-\infty},0[)={]{-\infty},0[}$. Consequently, $\beta_{P,Q}^a(]{-\infty},p[)={]{-\infty},p[}$ and $\beta_{P,Q}^a({]{-\infty},p]})={]{-\infty},p]}$.

\vspace{1mm}
From the previous analysis we infer that $\beta_{P,Q}^a(\Mm_a)={]{-\infty},p[}\cup{]q,+\infty[}=\Mm_a=\Ss_a$ and
$$
\beta_{P,Q}^a(\Nn_a)=
\left\{\begin{array}{cl}
]{-\infty}, p]\cup{[q,+\infty[}=\Tt_a&\text{ if $\Nn_a={]{-\infty}, p]}\cup{[q,+\infty[}$,}\\[4pt]
\R=\Tt_a&\text{ if $\Nn_a=\R$}.
\end{array}\right.
$$

\vspace{2mm} 
\noindent{\bf Case 3}: If $a_{n-1}>0$ and $\Mm_a={]{-\infty}, p[}\cup{]q,+\infty[}$ with $p\le q$, we have, since $\beta_{P,Q}^a$ is a regular function of odd degree greater than or equal to $1$ with negative leading coefficient (since we have excluded by hypothesis the possibility $P(a,\x_n)\equiv 0$), that 
$$
\begin{array}{lll}
\beta_{P,Q}^a(p)=p, &\quad&\displaystyle\lim_{t\to {-\infty}}\beta_{P,Q}^a(t)=+\infty;\\[8pt]
\beta_{P,Q}^a(q)=q\ge p, &\quad&\displaystyle\lim_{t\to +\infty}\beta_{P,Q}^a(t)={-\infty}.
\end{array}
$$
Once more by continuity, $\beta_{P,Q}^a(\Nn_a)=\beta_{P,Q}^a(\Mm_a)=\R$ if $p<q$. For the case $p=q$ a detailed analysis of the derivative of $\beta_{P,Q}^a$ (with respect to $\x_n$) shows that also in this case  $\beta_{P,Q}^a(\Nn_a)=\beta_{P,Q}^a(\Mm_a)=\R$ holds. Indeed, by Lemma \ref{bp}, $\frac{\partial\beta_{P,Q}}{\partial\x_n}(a,p)=1$; hence, $\beta_{P,Q}^a$ is an increasing function on a neighborhood of $p=q$. Thus, there is $\veps>0$ such that $\beta_{P,Q}^a(p-\veps)<\beta_{P,Q}^a(p)<\beta_{P,Q}^a(p+\veps)$, and, taking into account the behavior of the function $\beta_{P,Q}^a(t)$ when $t\rightarrow\pm\infty$, we have
\begin{gather*}
]{-\infty},\beta_{P,Q}^a(p)]\subset \beta_{P,Q}^a([p+\veps,+\infty[)\quad\text{and}\quad{[\beta_{P,Q}^a(p),+\infty[}\subset\beta_{P,Q}^a(]{-\infty},p-\veps]);
\end{gather*}
therefore, $\beta_{P,Q}^a(\Nn_a)=\beta_{P,Q}^a(\Mm_a)=\R=\Ss_a=\Tt_a$, and equalities \eqref{recort} are proved.
\end{proof}

\subsection{Placing a convex polyhedron in first trimming position}

Each bounded convex polyhedron is clearly in strong first trimming position; hence, our problem concentrates on placing unbounded convex polyhedra in first trimming position. In \cite{fu2} we analyze this fact carefully and we show that each $3$-dimensional unbounded convex polyhedron of $\R^3$ can be placed in strong first trimming position; for $n\geq4$ the situation is more delicate and we prove there that there are $n$-dimensional unbounded convex polyhedra that cannot be placed in strong first trimming position. Nevertheless, as we see below it is always possible to place an $n$-dimensional unbounded nondegenerate convex polyhedron in weak first trimming position. Since we are interested in developing a general wining strategy for every dimension, we will focus on the representation of complements of convex polyhedra and their interiors as regular images of $\R^n$.

\begin{lem}[Weak first trimming position for convex polyhedra]\label{colocarcasipmn}
Let $\pol\subset\R^n$ be an $n$-dimensional nondegenerate convex polyhedron and let $\Ff$ be one of its facets. Then  after a change of coordinates, we may assume that $\pol$ is in extreme weak first trimming position with respect to $\Ff$.
\end{lem}
\begin{proof}
As usual, we denote by $H$ the hyperplane generated by $\Ff$ and by $H^+$ the half-space of $\R^n$ with boundary $H$ that contains $\pol$.

If $\pol$ is bounded, it is enough to consider a change of coordinates that transforms the hyperplane $H$ onto $\{x_{n-1}=0\}$ and the half-space $H^+$ onto $\{x_{n-1}\leq0\}$. Thus, we assume in what follows that $\pol$ is unbounded. 

Let $\{H_1,\ldots,H_m\}$ be the minimal presentation of $\pol$. Since $\pol$ is nondegenerate, $m\geq n$. After a change of coordinates we may assume that $\bigcap_{i=1}^nH_i=\{{\bf 0}\}$ is a vertex of $\pol$ and $H_i^+=\{x_i\geq0\}$ for $1\leq i\leq n$. Consequently, if $\vec{e}_i:=(0,\ldots,0,1^{(i)},0,\ldots,0)\in\R^n$ denotes the vector whose coordinates are all zero except for the $i$th which is one, we have
$$
\pol\subset\{x_1\geq0,\ldots,x_n\geq0\}=\{{\bf 0}\}+\Big\{\sum_{i=1}^n\lambda_i\vec{e}_i:\ \lambda_i\geq0\Big\}\subset\{x_1+\cdots+x_n>0\}\cup\{{\bf 0}\}.
$$
We apply now a new change of coordinates $g:\R^n\to\R^n$ which transforms the half-space $\{x_1+\cdots+x_n\geq0\}$ onto $\{h:=x_n\geq0\}$ and keeps fixed the vertex ${\bf 0}$; for simplicity, we keep the notations $\pol$ for $g(\pol)$, $H_i$ for $g(H_i)$ and $H_i^+$ for $g(H_i^+)$. Note that $h({\bf0})<h(x)$ for all $x\in\pol\setminus\{{\bf0}\}$. Consider the basis $\bs:=\{\vec{u}_1:=\vec{g}(\vec{e}_1),\ldots,\vec{u}_n:=\vec{g}(\vec{e}_1)\}$ of $\R^n$ and its dual basis $\bs^*:=\{\vec{\ell}_1,\ldots,\vec{\ell}_n\}$; denote $\ell_i:=0+\vec{\ell}_i$. After the change of coordinates $g$, we have $H_i^{+}=\{\ell_i\geq0\}$ for $i=1,\ldots,n$ and 
$$
\pol\subset\Rr:=\{\ell_1\geq0,\ldots,\ell_n\geq0\}=\{{\bf 0}\}+\Big\{\sum_{i=1}^n\lambda_i\vec{u}_i:\ \lambda_i\geq0\Big\}\subset\{h:=x_n>0\}\cup\{{\bf 0}\}.
$$
Let us see now that the intersections $\Tt_c:=\pol\cap\Pi_c^-$ of $\pol$ with half-spaces $\Pi_c^-:=\{x_n\leq c\}$ are bounded. For our purposes it is enough to show that each intersection $\Ss_c:=\Rr\cap\Pi_c^-$ is a bounded subset of $\R^n$. If $c<0$, the intersection is empty and so we may assume that $c\geq0$. Write $\vec{u}_k:=(u_{1k},\ldots,u_{nk})$ for $k=1,\ldots,n$ and observe that $0=h({\bf0})<h({\bf0}+\vec{u}_k)=u_{nk}$. If $y:=(y_1,\ldots,y_n)\in \Ss_c$, there are $\lambda_1,\ldots,\lambda_n\geq0$ such that $y={\bf 0}+\lambda_1\vec{u}_1+\cdots+\lambda_n\vec{u}_n$. Then  since $y\in\Pi_c^-$,
$$
y_n=\lambda_1u_{n1}+\cdots+\lambda_nu_{nn}\leq c\ \Longrightarrow\ 0\leq\lambda_k\leq\max\Big\{\frac{c}{u_{nk}}:\ k=1,\ldots,n\Big\}=:M_c
$$ 
for $k=1,\ldots,n$; hence,
$$
\|y\|\leq\lambda_1\|u_1\|+\cdots+\lambda_n\|u_n\|\leq M_c(\|u_1\|+\cdots+\|u_n\|),
$$
and so $\Tt_c\subset \Ss_c$ is bounded.

Observe now that the hyperplane $H$ that contains $\Ff$ is not parallel to $\{x_n=0\}$ because otherwise $H^+:=\{x_n\leq c\}$ (because ${\bf0}\in\pol\subset H^+$) and $\pol\subset \Ss_c$ would be bounded, a contradiction. Thus, we choose a nonzero vector $\vec{v}\in\vec{\Pi}_0\cap\vec{H}$. Since the intersections $\p_c:=\pol\cap\{x_n=c\}\subset \Ss_c$ are bounded, the intersections of the parallel lines to the vector $\vec{v}$ with $\pol$ are also bounded. Next, we perform an additional change of coordinates that transforms: (1) $\vec{v}$ in $\vec{e}_n$, (2) $H$ in $\{x_{n-1}=0\}$, and (3) $H^+$ in $\{x_{n-1}\leq0\}$. 

In the actual situation, for each $a\in\R^{n-1}$ the fiber $\Ii_a:=\pi_n^{-1}(a,0)\cap\pol$ is bounded, $\Ff=\pol\cap\{x_{n-1}=0\}$ and $\pol\subset\{x_{n-1}\leq0\}$; hence, $\pol$ is in extreme weak first trimming position with respect to $\Ff$.
\end{proof}

\begin{lem}[Removing facets \text{and} first trimming position]\label{lem:ppr} 
Let $\pol\subset\R^n$ be an $n$-dimen\-sion\-al convex polyhedron and let $\{H_1,\ldots,H_m\}$ be the minimal presentation of $\pol$. Suppose that $\pol$ is in extreme weak first trimming position with respect to $\Ff_1:=\pol\cap H_1$. Then  $\pol_{1,\times}:=\bigcap_{j=2}^mH_j^+$ is in weak first trimming position. Moreover, if $\pol$ is in strong first trimming position so is $\pol_{1,\times}$ and if such position of $\pol$ is optimal so is the one of $\pol_{1,\times}$.
\end{lem}
\begin{proof} 
Suppose for a while that we have already proved that $\pol_{1,\times}$ is in weak first trimming position. Note that $\pol=\pol_{1,\times}\cap\{x_{n-1}\leq0\}$ because $H_1^+:=\{x_{n-1}\leq0\}$. Therefore, ${\mathfrak A}_\pol={\mathfrak A}_{\pol_{1,\times}}$ (see Definition \ref{def5a}); hence ${\mathfrak A}_\pol$ is either bounded or empty if and only if so is ${\mathfrak A}_{\pol_{1,\times}}$. Consequently, if $\pol$ is in strong first trimming position so is $\pol_{1,\times}$ and if such position of $\pol$ is optimal so is the one of $\pol_{1,\times}$.

Thus, it only remains to prove that the convex polyhedron $\pol_{1,\times}$ is in weak first trimming position, that is, for each $a:=(a_1,\ldots,a_{n-1})\in\R^{n-1}$ the intersection $\Ii_a':=\pol_{1,\times}\cap\pi_n^{-1}(a,0)$ is bounded. Observe first that if $a_{n-1}\leq0$, then 
$$
\Ii_a'=\pol_{1,\times}\cap\pi_n^{-1}(a,0)=\pol\cap\pi_n^{-1}(a,0),
$$
which is bounded because $\pol$ is in weak first trimming position. Suppose now that $a_{n-1}>0$. If $\Ii_a'=\pol_{1,\times}\cap\pi_n^{-1}(a,0)$ is unbounded, it is a half-line, say $\pol_{1,\times}\cap\pi_n^{-1}(a,0):=q\vec{v}$ where $\vec{v}:=\pm\vec{e}_n$. Let $p:=(p_1,\ldots,p_n)\in\pol\subset\pol_{1,\times}$; and let us check that the half-line $p\vec{v}$ is contained in $\pol_{1,\times}$.

Indeed, we may assume that $p\not\in q\vec{v}\cup q(-\vec{v})$ because otherwise, since $\pol_{1,\times}$ is convex, it is clear that $p\vec{v}\subset\pol_{1,\times}$.

Let $x\in p\vec{v}\setminus\{p\}$ and let $\mu>0$ be such that $x-p=\mu\vec{v}$. Choose a sequence $\{y_k\}_{k\in\N}\subset q\vec{v}\setminus\{q\}$ such that the sequence $\{\|y_k-q\|\}_{n\in\N}$ tends to $+\infty$. For each $k\in\N$ let $\rho_k>0$ be such that $y_k-q=\rho_k\vec{v}=\lambda_k(x-p)$, where $\lambda_k:=\rho_k/\mu>0$. Note that the sequence $\{\lambda_k=\|y_k-q\|/\|x-p\|\}_{k\in\N}$ also tends to $+\infty$. Since $p,y_k\in\pol_{1,\times}$ and this polyhedron is convex, each segment $\ol{py_k}$ is contained in $\pol_{1,\times}$. Now, since $y_k=q+\lambda_k(x-p)$, we have
$$
z_k:=\Big(\frac{1}{1+\lambda_k}\Big)q+\Big(\frac{\lambda_k}{1+\lambda_k}\Big)x=\Big(\frac{\lambda_k}{1+\lambda_k}\Big)p+\Big(\frac{1}{1+\lambda_k}\Big)y_k\in\ol{py_k}\subset\pol_{1,\times}.
$$
Thus, since $\pol_{1,\times}$ is closed, we deduce that
$$
x=\lim_{k\to\infty}\Big\{\Big(\frac{\lambda_k}{1+\lambda_k}\Big)x+\Big(\frac{1}{1+\lambda_k}\Big)q\Big\}_{k\in\N}=\lim_{k\to\infty}\{z_k\}_{k\in\N}\in\pol_{1,\times},
$$
and the half-line $ p\vec{v}$ is contained in $\pol_{1,\times}$. Moreover, since $p_{n-1}\leq0$, we get $p\vec{v}\subset\{x_{n-1}\leq0\}$, that is, $p\vec{v}\subset\pol_{1,\times}\cap\{x_{n-1}\leq0\}=\pol$. If we take now $a:=(p_1,\dots,p_{n-1})$, we deduce
$$
p\vec{v}\subset\Ii_a=\pol\cap\pi_n^{-1}(a,0);
$$
but the latter is a bounded set, because $\pol$ is in weak first trimming position, a contradiction. Thus, $\Ii_a'$ is bounded for each $a\in\R^{n-1}$ and so $\pol_{1,\times}$ is also in weak first trimming position, as wanted. 
\end{proof}

\section{Complement of a convex polyhedron}\label{s5}

The purpose of this section is to prove the part of Theorem \ref{mainext} concerning complements of convex polyhedra. More precisely, we prove the following:
\begin{thm}\label{redexta}
Let $n\geq2$ and let $\pol\subset\R^n$ be an $n$-dimensional convex polyhedron which is not a layer. We have: 
\begin{itemize}
\item[(i)] If $\pol$ is bounded, then $\R^n\setminus\pol$ is a polynomial image of $\R^n$. 
\item[(ii)] If $\pol$ is unbounded, then $\R^n\setminus\pol$ is a regular image of $\R^n$.
\end{itemize}
\end{thm}

The proof of the previous results runs by induction on the number of facets of $\pol$ in case (i) and by double induction on the number of facets and the dimension of $\pol$ in case (ii). The first step of the induction process in the second case concerns dimension $2$, which has being already approached by the second author in \cite{u2} with similar but simpler techniques. We approach first the following particular case.

\begin{lem}[The orthant]\label{lem2} 
Let $\pol\subset\R^n$ be an $n$-dimensional nondegenerate convex polyhedron with $n\geq2$ facets. Then  $\R^n\setminus \pol$ is a polynomial image of $\R^n$. 
\end{lem}
\begin{proof} 
After a change of coordinates it is enough to prove by induction on $n$, that $\pp(\R^n\setminus\ol{\Qq}_n)=n$ for each $n\geq 2$, where 
$$
\ol{\Qq}_n=\{x_1\ge0,\dots,x_n\ge 0\}\subset\R^n.
$$ 
The case $n=2$ is proved in \cite[Lem. 3.11]{u2}. Suppose that the statement is true for $n-1$ (with $n\geq3$), and we will see that it also holds for $n$. By induction hypothesis, there is a polynomial map $f:\R^{n-1}\to\R^{n-1}$ such that $f(\R^{n-1})=\R^{n-1}\setminus\ol{\Qq}_{n-1}$. Define 
$$
\widehat{f}:\R^n\to\R^n,\, (x_1,\dots,x_n)\mapsto(f(x_1,\dots,x_{n-1}),x_n),
$$ 
that satisfies the equality 
$$
\widehat{f}(\R^n)=f(\R^{n-1})\times\R=(\R^{n-1}\setminus\ol{\Qq}_{n-1})\times\R.
$$ 
Let $g_1:\R^n\to\R^n$ be a change of coordinates such that $g_1((\R^{n-1}\setminus\ol{\Qq}_{n-1})\times\R)=\R^n\setminus\pol_1$, where
$$
\pol_1=\{x_1\ge0,\dots,x_{n-3}\ge0,\ x_{n-2}+x_n\ge 0,x_{n-2}-x_n\ge0\}\subset\R^n
$$
is a convex polyhedron in optimal strong first trimming position. 

Indeed, let $a=(a_1,\ldots,a_{n-1})\in\R^{n-1}$ be such that the line $\ell_a:=\{(a_1,\ldots,a_{n-1},t):\ t\in\R\}$ intersects $\pol_1$; then $a_1,\ldots,a_{n-3},a_{n-2}\geq0$ (because $a_{n-2}+t\geq0$ and $a_{n-2}-t\geq0$) and $-a_{n-2}\leq t\leq a_{n-2}$. Thus, $\Ii_a$ is either the empty set or the bounded interval
$$
\Ii_a:=\pol\cap\ell_a=\{a\}\times[-a_{n-2},a_{n-2}]\quad(a_{n-2}\ge 0),
$$
which contains $(a,0)$; hence, $\pol_1$ is in optimal strong first trimming position.

By Proposition \ref{trim}, there is a polynomial map $h:\R^n\to\R^n$ such that $h(\R^n\setminus\pol_1)=\R^n\setminus(\pol_1\cap\{x_{n-1}\le 0\})$. Moreover, there is a change of coordinates $g_2:\R^n\to\R^n$ such that 
$$
g_2(\R^n\setminus(\pol_1\cap\{x_{n-1}\le 0\}))=\R^n\setminus\ol{\Qq}_n.
$$
Thus, the composition $g:=g_2\circ h\circ g_1\circ\widehat{f}:\R^n\to\R^n$ is a polynomial map satisfying $g(\R^n)=\R^n\setminus\ol{\Qq}_n$.
\end{proof}

\begin{proof}[Proof of Theorem \em\ref{redexta}(i)]
Consider first the case when $\pol:=\Delta$ is an $n$-simplex in $\R^n$. We may assume that we are working with the $n$-simplex $\Delta\subset\R^n$, which is in strong first trimming position because it is bounded, whose vertices are the points
$$
v_0:=(0,\dots,0),\ v_i:=(0,\dots,0,1^{(i)},0,\dots,0),\ \text{for}\ i\neq n-1\ \text{and}\ v_{n-1}:=(0,\dots,0,-1,0).
$$
Write $\Delta:=\bigcap_{i=1}^{n+1}{H_{i}'}^+$ where $\{H_1',\dots,H_{n+1}'\}$ is the minimal presentation of $\Delta$ and order the hyperplanes in such a way that $H_1':=\{x_{n-1}=0\}$; hence, $\Delta_{1,\times}:=\bigcap_{i=2}^{n+1}H_i^+$ is an $n$-dimensional nondegenerate convex polyhedron with $n$ facets. By Lemma \ref{lem2} there is a polynomial map $g:\R^n\to\R^n$ such that $g(\R^n)=\R^n\setminus\Delta_{1,\times}$. Since $\Delta$ is in extreme strong first trimming position with respect to the facet $\Delta\cap\{x_{n-1}=0\}$, we find by means of Lemmata \ref{trim} and \ref{lem:ppr} a polynomial map $h:\R^n\to\R^n$ such that $h(\R^n\setminus \Delta_{1,\times})=\R^n\setminus\Delta$; hence, the polynomial map $f:=h\circ g$ satisfies $f(\R^n)=\R^n\setminus\Delta$.

Let now $\pol$ be an $n$-dimensional bounded convex polyhedron of $\R^n$. We choose an $n$-simplex $\Delta\subset\R^n$ such that $\pol\subset\Int\Delta$. We write $\pol:=\bigcap_{i=1}^mH_i^+$ where $\{H_1,\dots,H_m\}$ is the minimal presentation of $\pol$, and consider the convex polyhedra $\{\pol_0,\dots,\pol_m\}$ defined by:
$$
\pol_0:=\Delta,\quad \pol_j:=\pol_{j-1}\cap H_{j}^+\ \text{for}\ 1\le j\le m.
$$
Thus, $\pol_{m}=\pol$ and for each $1\leq j\leq m$ the convex polyhedron $\pol_{j}=\Delta\cap(\bigcap_{i=1}^jH_i^+)$ has a facet $\Ff_{j}$ contained in the hyperplane $H_j$. Moreover, the set $\R^n\setminus\Delta$ is, as we have already seen above, the image of a polynomial map $f^{(0)}:\R^n\to\R^n$. 

For each index $1\leq j\leq m$ there is a change of coordinates $h^{(j)}:\R^n\to\R^n$ such that $h^{(j)}(H_j^+)=\{x_{n-1}\leq0\}$; let
$$
\pol_j':=h^{(j)}(\pol_j)=h^{(j)}(\pol_{j-1})\cap h^{(j)}(H_{j}^+)=h^{(j)}(\pol_{j-1})\cap\{x_{n-1}\leq0\}.
$$
By Lemma \ref{trim}, there is a polynomial map $g^{(j)}:\R^n\to\R^n$ such that
$$
g^{(j)}(\R^n\setminus h^{(j)}(\pol_{j-1}))=\R^n\setminus(h^{(j)}(\pol_{j-1})\cap \{x_{n-1}\le 0\})=\R^n\setminus\pol'_{j};
$$
hence, the polynomial map $f^{(j)}:=(h^{(j)})^{-1}\circ g^{(j)}\circ h^{(j)}$ satisfies $f^{(j)}(\R^n\setminus \pol_{j-1})=\R^n\setminus \pol_{j}$. 

Finally, $f:=f^{(m)}\circ \cdots \circ f^{(0)}:\R^n\to\R^n$ is a polynomial map which satisfies $f(\R^n)=\R^n\setminus\pol_{m}=\R^n\setminus\pol$, as wanted.
\end{proof}

\begin{proof}[Proof of Theorem \em\ref{redexta}(ii)]
We proceed by induction on the pair $(n,m)$ where $n$ denotes the dimension of $\pol$ and $m$ denotes the number of facets of $\pol$. As we have already commented the case $n=2$ has already been already solved in \cite[Thm.1]{u2} (using polynomial maps) and we do not include further details. Suppose in what follows that $n\geq3$ and that the result holds for any $k$-dimensional convex polyhedron of $\R^k$ which is not a layer, if $2\leq k\leq n-1$. We distinguish two cases:

\noindent{\bf Case 1.} \em $\pol$ is a degenerate convex polyhedron\em. By Lemma~\ref{lem:deg2}, we may assume that $\pol=\p\times\R^{n-k}$ where $1\le k<n$ and $\p\subset\R^k$ is a nondegenerate convex polyhedron. If $k=1$ then $\pol$ is either a layer, which disconnects $\R^n$ and cannot be a regular image of $\R^n$, or a half-space, whose complement can be expressed as a polynomial image of $\R^n$ (see \cite[Thm.1.6]{fg1}). We now assume $k>1$; by induction hypothesis $\R^k\setminus\p$ is the image of $\R^k$ by a regular map $f:\R^k\to\R^k$; hence, $\R^n\setminus\pol$ is the image of the regular map $(f,\id_{\R^{n-k}}):\R^k\times\R^{n-k}\to\R^k\times\R^{n-k}$.

\vspace{1mm}
\noindent{\bf Case 2.} \em $\pol$ is a nondegenerate convex polyhedron\em. If $\pol$ has $n$ facets (and so just one vertex), it is enough to apply Lemma \ref{lem2}. Thus, suppose that the number $m$ of facets of $\pol$ is strictly greater than $n$ and let $\{H_1,\ldots,H_m\}$ be the minimal presentation of $\pol:=\bigcap_{i=1}^mH_i^+$. 

By Lemma \ref{colocarcasipmn}, we may assume that $\pol$ is in extreme weak first trimming position with respect to $\Ff_1$. By Lemma \ref{lem:ppr}, the convex polyhedron $\pol_{1,\times}=\bigcap_{j=2}^mH_j^+$ is in weak first trimming position. This allows us to apply Proposition \ref{trim} to the convex polyhedron $\pol_{1,\times}$ and to find a regular map $f:\R^n\to\R^n$ such that 
$$
f(\R^n\setminus\pol_{1,\times})=\R^n\setminus(\pol_{1,\times}\cap\{x_{n-1}\le 0\})=\R^n\setminus(\pol_{1,\times}\cap H_1^+)=\R^n\setminus\pol.
$$
Now, $\pol_{1,\times}$ is a convex polyhedron with $m-1\geq n$ facets, and so it is not a layer; hence, by induction hypothesis there is a regular map $g:\R^n\to\R^n$ such that $g(\R^n)=\R^n\setminus\pol_{1,\times}$. We conclude that the polynomial map $h:=f\circ g$ satisfies 
$$
h(\R^n)=f(g(\R^n))=f(\R^n\setminus\pol_{1,\times})=\R^n\setminus\pol,
$$ 
and we are done. 
\end{proof}

\begin{remark}
The previous proof is constructive. Moreover, if the convex polyhedron that appears in a certain step of the inductive process is in extreme strong first trimming position with respect to a facet $\Ff_1$, then by Lemma \ref{trim} the regular map of this step can be chosen polynomial. In view of Definition \ref{1tm}, this reduces the complexity of the final regular map. Moreover, if this can be done in each inductive step, then the final regular map can be assumed to be polynomial.
\end{remark}

\section{Trimming tools of type two}\label{s6}

In this section we develop the main tools that will be used in Section \ref{s7} to prove the part of Theorem \ref{mainext} concerning complements of interiors of convex polyhedra. We will use freely the preliminaries already introduced in \S\ref{lonchas} and Lemma \ref{prelonchas1}. We begin with the definition of second trimming position.

\begin{defines}\label{def5b}
(1) An $n$-dimensional convex polyhedron $\pol\subset\R^n$ is in \em weak second trimming position \em if: 
\begin{itemize}
\item[(i)] For each $a\in\R^{n-1}$ the set $\Ii_a:=\pol\cap\pi_n^{-1}(a,0)$ is upperly bounded (here we endow the line $\pi_n^{-1}(a,0)$ with the orientation induced by the vector $\vec{e}_n$).
\end{itemize}
We say that $\pol$ is in \em strong second trimming position \em if it also satisfies:
\begin{itemize}
\item[(ii)] The set ${\mathfrak B}_\pol:=\{a\in\R^{n-1}:\ \Ii_a\neq\varnothing,\ (a,0)\not\in\Ii_a\}$ is bounded.
\end{itemize}
Moreover, if ${\mathfrak B}_\pol=\varnothing$  we say that the second trimming position of $\pol$ is \em optimal\em.

(2) On the other hand, we say that the (either weak or strong) second trimming position of $\pol\subset\R^n$ is \em extreme with respect to its facet $\Ff$ \em if $\Ff\subset\{x_n=0\}$ and $\pol\subset\{x_n\leq0\}$.
\end{defines}

\begin{figure}[t]
\begin{minipage}[b]{0.26\linewidth}
\begin{center}
\begin{tikzpicture}[line join=round,background rectangle/.style=
{double,draw=gray},
show background rectangle,xscale=0.456, yscale=0.456]
\filldraw[fill=gray!80,fill opacity=0.8, draw=gray,dashed](-3,-4)--(0,-2)--(0,-1)--(-1,0)--(-3,0)--cycle;
\draw[ultra thick,<->] (-3,-4)--(0,-2)--(0,-1)--(-1,0)--(-3,0);
\draw (-3,-2) node[right]{$\EuScript K$};
\draw[<->] (0,-5) -- (0,3) node[right] {$x_2$};
\draw[<->] (3,0) -- (-4,0) node[above] {$x_1$};
\draw[densely dashed,very thin] (-0.5,-2.3)--(0.05,-2.3) node[right]{$p$};
\draw[densely dashed,very thin] (-0.5,-0.5)--(0.05,-0.5) node[right]{$q$};
\path (-0.5,-1.3) node[left]{$\EuScript I_a$};
\draw[very thick, gray] (-0.5,-2.3)--(-0.5,-0.5);
\filldraw[fill=gray,draw=gray] (-0.5,-2.3) circle (0.5mm);
\filldraw[fill=gray,draw=gray] (-0.5,-0.5) circle (0.5mm);
\end{tikzpicture}\\
(a)\vspace{3pt}\\
\begin{tikzpicture}[line join=round,background rectangle/.style=
{double,draw=gray},
show background rectangle,xscale=0.456, yscale=0.456]
\filldraw[fill=gray!80,fill opacity=0.8,draw=gray,dashed](-3,-4)--(0,-2)--(0,0)--(-3,0)--cycle;
\draw[ultra thick,<->] (-3,-4)--(0,-2)--(0,0)--(-3,0);
\draw (-3,-2) node[right]{$\EuScript K$};
\draw[<->] (0,-5) -- (0,3) node[right] {$x_2$};;
\draw[<->] (3,0) -- (-4,0) node[above] {$x_1$};;
\draw[densely dashed,very thin] (-0.5,-2.3)--(0.05,-2.3) node[below right]{$p$};
\draw[densely dashed,very thin] (-0.5,0)--(0.05,0) node[below right]{$q$};

\draw[very thick, gray] (-0.5,-2.3)--(-0.5,0.05);
\filldraw[fill=gray,draw=gray] (-0.5,-2.3) circle (0.5mm);
\filldraw[fill=gray,draw=gray] (-0.5,0) circle (0.5mm);
\path (-0.5,-1.3) node[left]{$\EuScript I_a$};
\end{tikzpicture}\\
(b)
\end{center}
\label{fig:figure2ab}
\end{minipage}
\hspace{0.5cm}
\begin{minipage}[b]{0.55\linewidth}
\begin{center}
\begin{tikzpicture}[line join=round,background rectangle/.style=
{double,draw=gray},
show background rectangle, loose background, xscale=0.645, yscale=0.53]
\filldraw[fill=blue!40,fill opacity=0.4,draw=none](0,3.352)--(3.637,4.041)--(3.637,-3.781)--(0,-4.469)--cycle;
\draw[very thin](0,3.352)--(3.637,4.041)--(3.637,-3.781)--(0,-4.469);
\draw[thick,arrows=<->](3.03,-4.082)--(0,0)--(0,4.469);
\filldraw[fill=green!40,fill opacity=0.4,very thin](-1.364,-2.633)--(-1.364,5.189)--(2.727,-.321)--(2.727,-8.143)--cycle;
\draw[thick,arrows=->](0,0)--(-7.273,-1.378);
\filldraw[fill=blue!40,fill opacity=0.4,draw=none](-6.061,2.204)--(0,3.352)--(0,-4.469)--(-6.061,-5.617)--cycle;
\draw[very thin](0,-4.469)--(-6.061,-5.617)--(-6.061,2.204)--(0,3.352);
\filldraw[draw=black,ultra thick,fill=black!60,fill opacity=0.8](-6.061,-4.128)--(-6.061,-1.148)--(-1.616,-.306)--(0,-1.117)--cycle;
\filldraw[draw=black,ultra thick,fill=black!60,fill opacity=0.8](-6.061,-4.128)--(0,-1.117)--(2.727,-7.026)--(-3.334,-10.036)--cycle;
\filldraw[draw=black,ultra thick,fill=black!60,fill opacity=0.8](2.727,-7.026)--(0,-1.117)--(.455,-.612)--(2.727,-3.673)--cycle;
\filldraw[draw=black,ultra thick,fill=black!60,fill opacity=0.8](0,-1.117)--(-1.616,-.306)--(.455,-.612)--cycle;

\filldraw[draw=black!60,dashed,ultra thin,fill=gray!40,fill opacity=0.3](-3.334,-10.036)--(-6.061,-4.128)--(-6.061,-1.148)--(-3.334,-4.821)--cycle;
\filldraw[draw=black,ultra thick,fill=black!60,fill opacity=0.8](-6.061,-1.148)--(-1.616,-.306)--(.455,-.612)--(2.727,-3.673)--(-3.334,-4.821)--cycle;
\filldraw[draw=black!60,dashed,ultra thin,fill=gray!40,fill opacity=0.3](-3.334,-10.036)--(-3.334,-4.821)--(2.727,-3.673)--(2.727,-7.026)--cycle;
\draw[dashed,ultra thick,draw=white!60](-6.061,-4.128)--(-6.061,-1.148)--(-3.334,-4.821)--(2.727,-3.673)--(2.727,-7.026)--(-3.334,-10.036)--(-6.061,-4.128);
\path (3.03,-4.082) node[below right]{$x_2$} (0,4.469) node[above]{$x_3$} (-7.273,-1.378) node[above left]{$x_1$};
\path (-2.9,5.4) node[right]{$x_1=0$};
\path (-6.061,2.8) node[right]{$x_2=0$};
\path (-3.4,-8.622) node[right]{$\EuScript K$};
\path (-6.061,-2.638) node[left]{$\EuScript F_1$};
\path (-3.2,-4.4) node[above]{$\EuScript F$};
\path (2.727,-5.536) node[right]{$\EuScript F_2$};
\end{tikzpicture}\\
(c)
\end{center}
\label{fig:figure2c}
\end{minipage}

\vspace{-5mm}
\caption{{\small (a) and (b) respectively show convex polygons in nonoptimal and optimal strong second trimming position. (c) shows a convex polyhedron in strong second trimming position.}}\label{fig:2}
\end{figure}
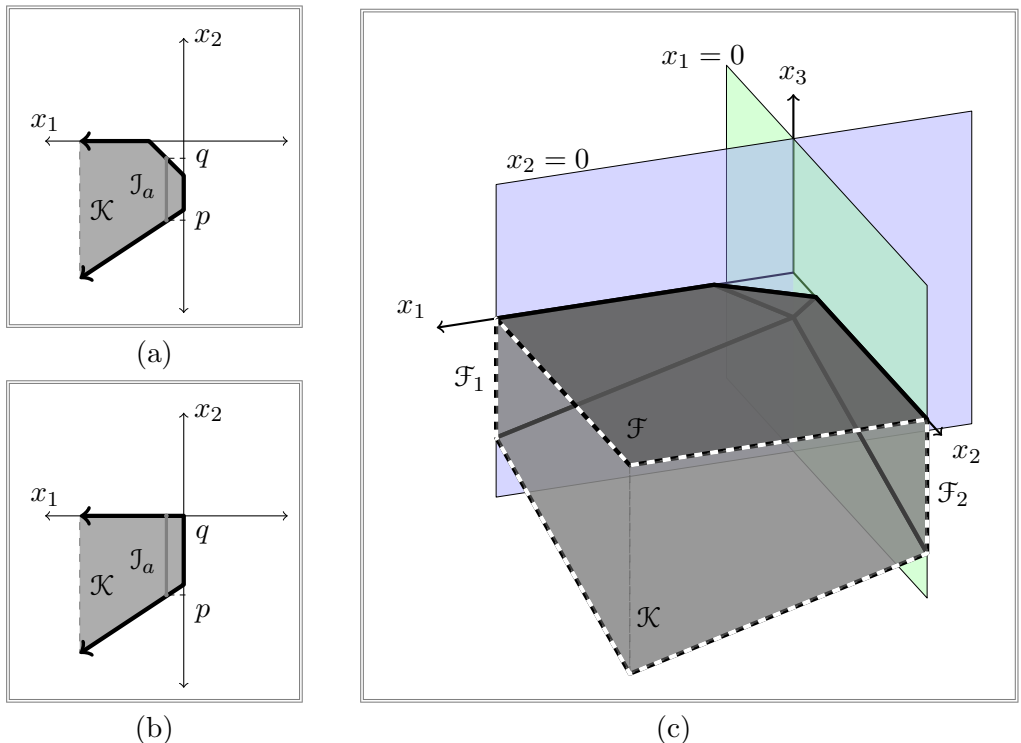

\begin{remarks}\label{bajuste}
(1) Each bounded convex polyhedron is clearly in strong second trimming position; hence, the problem of placing a convex polyhedron in second trimming position concentrates on placing those which are unbounded. In \cite{fu2} we analyze this fact carefully and we show that  each $3$-dimensional polyhedron can be placed in strong second trimming position; for $n\geq4$ the situation is more delicate and we prove that there are $n$-dimensional unbounded convex polyhedra that cannot be placed in strong second trimming position. Nevertheless, notice that condition (2) in Definition \ref{def5b} imply condition (1.i). Thus, all $n$-dimensional convex polyhedron $\pol\subset\R^n$ contained in $\{x_n\leq0\}$ and such that $\Ff:=\pol\cap\{x_n=0\}$ is a facet of $\pol$ is in extreme weak second trimming position with respect to its facet $\Ff$. 

(2) Let $\pol\subset\R^n$ be an $n$-dimensional convex polyhedron in strong second trimming position. Then  there is $N>0$ such that 
$$
B:=\bigcup_{a\in{\mathfrak B}_\pol}\partial\Ii_a\subset{\mathfrak B}_\pol\times[-N,N].
$$

This proof runs analogously to that of Lemma \ref{aajuste} and we do not include it. The fact that for the strong second trimming position each set $\Ii_a$ is just upperly bounded but not necessarily bounded does not change the proof in a relevant way.
\end{remarks}

\begin{define}\label{2tm}
Let $M>0$ and let $P\in\R[\x]$ be a polynomial such that $d:=\deg(P)=\deg_{\x_n}(P)\geq1$. We write shortly ${\rm p}:=$ polynomic and ${\rm r}:=$ regular, and consider the function
$$
\gamma_{P,M}^*(\x):=
\begin{cases}
\x_n(1-\x_n\frac{P^2(\x)}{M})^2&\text{if $\ast={\rm p}$,}\\[4pt]
\x_n\frac{(1-\x_nP^2(\x))^2}{(1-\x_nP^2(\x))^2+\frac{\x_n^2P^2(\x)}{(M(1+\|\x\|^2))^{2d+2}}}&\text{if $\ast={\rm r}$,}
\end{cases}
$$
(which is either polynomial or regular depending on the case) and we define the \emph{trimming map $\TS_{P,M}^*$ of type} II associated to $P$ and $M$ as the either polynomial or regular map (depending on $\ast={\rm p}$ or $\ast={\rm r}$)
$$
\TS_{P,M}^*:\R^n\to\R^n,\ x=(x_1,\dots,x_n)\mapsto(x_1,\dots,x_{n-1},\gamma_{P,M}^*(x)).
$$ 
Observe that in any case $\gamma_{P,M}^*$ is a regular map of degree $\deg(\gamma_{P,M}^*)=\deg_{\x_n}(\gamma_{P,M}^*)\geq1$.
\end{define}
\begin{remarks}
The following properties are straightforward:
\begin{itemize}
\item[(i)] For each $a\in\R^{n-1}$, the regular function $\gamma^{\ast,a}_{P,M}(\x_n):=\gamma_{P,M}^*(a,\x_n)\in\R(\x_n)$ (which depends only on the variable $\x_n$) has odd degree greater than or equal to $1$; hence, $\gamma^{\ast,a}_{P,M}(\R)=\R$ for each $a\in\R^{n-1}$.
\item[(ii)] $\TS^*_{P,M}(\pi_n^{-1}(a,0))=\pi_n^{-1}(a,0)$ for each $a\in\R^{n-1}$ and so $\TS^*_{P,M}(\R^n)=\R^n$.
\item[(iii)] The set of fixed points of $\TS^*_{P,M}$ contains the set $\{x_nP(x)=0\}$.
\end{itemize}
\end{remarks}

Next, we present some preliminary results to prove Lemma \ref{trim2}, which is the counterpart of Lemma \ref{trim}, concerning this time the second trimming position.

\begin{lem}\label{acotacionpm}
Let $f,\veps:\R^n\to\R$ be differentiable functions whose common zero-set $\{f=0,\veps=0\}$ is empty. Write $g:=\x_n\frac{f^2}{f^2+\veps^2}$ and suppose that for each $x\in\{x_n\leq0\}$ the following inequalities hold:
\begin{itemize}
\item[(i)] $|f(x)|\geq1$ and $|x_n||\frac{\partial f}{\partial\x_n}(x)|\veps^2(x)<\frac{1}{16}$;
\item[(ii)] $|x_n\veps(x)|<\frac{1}{8}$, $|\veps(x)|<\frac{1}{2}$ and $|\frac{\partial\veps}{\partial\x_n}(x)|<\frac{1}{4}$.
\end{itemize}
Then  $\frac{\partial g}{\partial\x_n}(x)>\frac{1}{2}$ for each $x\in\{x_n\leq0\}$ and $\frac{\partial g}{\partial\x_n}(x)=1$ for each $x\in\{\veps=0\}$.
\end{lem}
\begin{proof}
We write $h:=\frac{\veps^2}{f^2+\veps^2}$ and observe that $g=\x_n(1-h)$,  $\frac{\partial g}{\partial\x_n}=1-h-\x_n\frac{\partial h}{\partial\x_n}$, and
$$
\frac{\partial h}{\partial\x_n}=\frac{2\veps}{f^2+\veps^2}\frac{\partial\veps}{\partial\x_n}-\frac{\veps^2}{(f^2+\veps^2)^2}\Big(2f\frac{\partial f}{\partial\x_n}+2\veps\frac{\partial\veps}{\partial\x_n}\Big).
$$
Notice that  $\frac{\partial g}{\partial\x_n}(x)=1$ for each $x\in\{\veps=0\}$. Using the inequality
$$
\Big|\x_n\frac{\partial h}{\partial\x_n}\Big|\leq2|\x_n\veps|\Big|\frac{\partial\veps}{\partial\x_n}\Big|\frac{1}{(f^2+\veps^2)}+2|\x_n|\Big|\frac{\partial f}{\partial\x_n}\Big|\veps^2\frac{|f|}{(f^2+\veps^2)^2}+2|\x_n\veps|\Big|\frac{\partial\veps}{\partial\x_n}\Big|\frac{\veps^2}{(f^2+\veps^2)^2}
$$
and taking into account for $x\in\{x_n\leq0\}$  the inequalities in the statement of the Lemma we get
$$
\Big|x_n\frac{\partial h}{\partial\x_n}(x)\Big|\leq2\cdot\frac{1}{8}\cdot\frac{1}{4}\cdot1+2\cdot\frac{1}{16}\cdot1\cdot1+2\cdot\frac{1}{8}\cdot\frac{1}{4}\cdot\frac{1}{4}=\frac{13}{64}<\frac{1}{4}.
$$
Thus,
$$
\frac{\partial g}{\partial\x_n}(x)=1-h(x)-x_n\frac{\partial h}{\partial\x_n}(x)\geq1-\Big|\frac{\veps^2(x)}{\veps^2(x)+f^2(x)}\Big|-\Big|x_n\frac{\partial h}{\partial\x_n}(x)\Big|>1-\frac{1}{4}\cdot1-\frac{1}{4}=\frac{1}{2},
$$
as wanted.
\end{proof}

\begin{lem}\label{acotacionpm2}
Let $P\in\R[\x]$ be a polynomial of degree $d\geq1$, and set $f:=1-\x_nP^2$, $\veps:=\veps_M:=\frac{\x_nP}{M(1+\|\x\|^2)^{d+1}}$. Then  there is $M_0>0$ such that for $M\geq M_0$ and $x\in\R^n$ the following inequalities hold:
\begin{itemize}
\item[(i)] $|x_n||\frac{\partial f}{\partial\x_n}(x)|\veps^2(x)<\frac{1}{16}$;
\item[(ii)] $|x_n\veps(x)|<\frac{1}{8}$, $|\veps(x)|<\frac{1}{2}$ and $|\frac{\partial\veps}{\partial\x_n}(x)|<\frac{1}{4}$.
\end{itemize}
\end{lem}
\begin{proof}
First, observe that
\[
\begin{array}{lll}
\displaystyle \deg\Big(\x_n^3\frac{\partial f}{\partial\x_n}P^2\Big)\leq 4d+3, &\qquad&\displaystyle\deg(\x_n^2P)=d+2,\\
\displaystyle\deg\Big(P+\x_n\frac{\partial P}{\partial\x_n}\Big)\leq d, &&\displaystyle \deg(\x_nP)=d+1.
\end{array}
\]
Moreover, we have
$$
\frac{\partial\veps}{\partial\x_n}=\Big(P+\x_n\frac{\partial P}{\partial\x_n}\Big)\frac{1}{M(1+\|\x\|^2)^{d+1}}-\frac{2(d+1)\x_n^2P}{M(1+\|\x\|^2)^{d+2}}
$$

By Lemma \ref{acotacion}, there is $M_0>0$ such that if $M\geq M_0$, we have:
\begin{center}
\begin{tabular}{rrl}
$\Big|\x_n^3\dfrac{\partial f}{\partial\x_n}P^2\Big|<\dfrac{M}{16}(1+\|\x\|^2)^{2d+2}$& $\Longrightarrow$&$|\x_n|\Big|\dfrac{\partial f}{\partial\x_n}\Big|\veps^2<\dfrac{1}{16}$,\\[8pt]
$|\x_n^2P|<\dfrac{M}{8}(1+\|\x\|^2)^{d+1}$&$\Longrightarrow$&$|\x_n\veps|<\dfrac{1}{8}$,\\[8pt]
$|\x_nP|<\dfrac{M}{2}(1+\|\x\|^2)^{d+1}$&$\Longrightarrow$&$|\veps|<\dfrac{1}{2}$,\\[8pt]
$\Big|P+\x_n\dfrac{\partial P}{\partial\x_n}\Big|<\dfrac{M}{8}(1+\|\x\|^2)^{d+1}$&\!\!\!\!\!\!\!\!\!\multirow{2}{*}{\Bigg\}\ $\Longrightarrow$}&\multirow{2}{*}{$\Big|\dfrac{\partial\veps}{\partial\x_n}(x)\Big|<\dfrac{1}{4},$}\\
$2(d+1)|\x_n^2P|<\dfrac{M}{8}(1+\|\x\|^2)^{d+2}$ 
\end{tabular}
\end{center}
and so all the inequalities in the statement hold.
\end{proof}

\begin{lem}\label{gp} 
Let $P\in\R[\x]$ be a nonzero polynomial such that $d:=\deg(P)=\deg_{\x_n}(P)$. Then 
\begin{itemize} 
\item[(i)] For each compact set $K\subset\R^n$ there is $M_K>0$ such that if $M\geq M_K$ the partial derivative $\frac{\partial\gamma_{P,M}^{\rm p}}{\partial\x_n}$ is strictly positive on $K\cap\{x_n\leq0\}$ and it is constantly $1$ on the set $\{x_nP(x)=0\}$.
\item[(ii)] There is $M_0>0$ such that if $M\geq M_0$ the partial derivative $\frac{\partial\gamma_{P,M}^{\rm r}}{\partial\x_n}$ is strictly positive on the set $\{x_n\leq0\}$ and it is constantly $1$ on the set $\{x_nP(x)=0\}$ for each $M>0$.
\end{itemize}
\end{lem}
\begin{proof}
To prove (i), observe first that
$$
\frac{\partial\gamma_{P,M}^{\rm p}}{\partial\x_n}=\Big(1-\x_n\frac{P^2(\x)}{M}\Big)
\Big(1-3\x_n\frac{P^2(\x)}{M}-4\x_n^2\frac{P(\x)}{M}\frac{\partial P}{\partial\x_n}(\x)\Big)
$$
Denote by $M_K\geq1$ the maximum of the continuous function 
$$
\R^n\to\R,\ x:=(x_1,\dots,x_n)\mapsto\Big|4x_n^2P(x)\frac{\partial P}{\partial\x_n}(x)\Big|
$$ 
on the compact set $K\cap\{x_n\leq0\}$. If $M\geq M_K$, then $\frac{\partial\gamma_{P,M}^{\rm p}}{\partial\x_n}$ is strictly positive over $K\cap\{x_n\leq0\}$. Moreover, since $\x_nP(\x)$ divides $\frac{\partial\gamma_{P,M}^{\rm p}}{\partial\x_n}(\x)-1$, the partial derivative $\frac{\partial\gamma_{P,M}^{\rm p}}{\partial\x_n}$ is constantly $1$ on the set $\{x_nP(x)=0\}$.

To prove (ii), note that $f:=1-\x_nP^2$ is greater than or equal to $1$ on the set $\{x_n\leq0\}$ and apply Lemmata \ref{acotacionpm} and \ref{acotacionpm2} to $f$, $\veps:=\frac{\x_nP(\x)}{(M(1+\|\x\|^2))^{d+1}}$ and $g:=\gamma_{P,M}^{\rm r}$.
\end{proof}

We are ready to prove the counterpart of Lemma~\ref{trim} for the second trimming position of a convex polyhedron.

\begin{lem}\label{trim2}
Let $\pol\subset\R^n$ be an $n$-dimensional convex polyhedron in either weak or strong second trimming position and let $\Ff_1,\ldots,\Ff_r$ be the facets of $\pol$ that are contained in hyperplanes nonparallel to the vector $\vec{e}_n$. Let $\Lambda$ be an arbitrary subset of the hyperplane $\Pi:=\{x_n=0\}$ and let $P\in\R[\x]$ be a nonzero polynomial that vanishes identically on the facets $\Ff_1,\ldots,\Ff_r$, with $\deg(P)=\deg_{\x_n}(P)$ \em (see Remark \ref{veryeasy})\em. We write 
$$
\ast=\begin{cases}
{\rm r}&\text{if the second trimming position of $\pol$ is only weak,}\\
{\rm p}&\text{if the second trimming position of $\pol$ is strong.}
\end{cases}
$$
Then  there exists $M>0$ such that 
$$
\mathsf{G}_{P,M}^{\ast}(\R^n\setminus\Int(\pol))=\mathsf{G}_{P,M}^{\ast}(\R^n\setminus(\Int(\pol)\cup\Lambda))=\R^n\setminus\Int_{\R^n}(\pol\cap\{x_n\leq0\}).
$$
\end{lem}
\begin{proof} 
First, for each $a\in\R^{n-1}$ consider the sets
\begin{equation*}
\begin{split}
\Mm_a:=\,&\{t\in\R:\ (a,t)\in\R^n\setminus\Int(\pol)\},\\[2pt]
\Nn_a:=\,&\{t\in\R:\ (a,t)\in\R^n\setminus(\Int(\pol)\cup\Lambda)\}=\left\{\begin{array}{cl}
\Mm_a&\ \text{ if $(a,0)\not\in\Lambda$,}\\
\Mm_a\setminus\{0\}&\ \text{ if $(a,0)\in\Lambda$}.
\end{array}\right.\\[2pt]
\Tt_a:=\,&\{t\in\R:\ (a,t)\in\R^n\setminus\Int(\pol\cap\{x_n\leq0\})\}=\Mm_a\cup[0,+\infty[.
\end{split}
\end{equation*}
From the previous definition the following equalities follow:
\begin{multline*}
\R^n\setminus\Int(\pol)=\bigsqcup_{a\in\R^{n-1}}\{a\}\times\Mm_a,\quad
\R^n\setminus(\Int(\pol)\cup\Lambda)=\bigsqcup_{a\in\R^{n-1}}\{a\}\times\Nn_a,\\
\text{and}\quad\R^n\setminus\Int(\pol\cap\{x_n\le 0\})=\bigsqcup_{a\in\R^{n-1}}\{a\}\times\Tt_a.
\end{multline*}\setcounter{substep}{0}

\begin{substeps}{trim2}\label{signder2}
At this point, we choose $M>0$ as follows. If $\pol$ is in strong second trimming position, ${\mathfrak B}_\pol$ is bounded and there exists, by Remark \ref{bajuste}(2), a real number $N>0$ such that $B:=\bigcup_{a\in{\mathfrak B}_\pol}\partial\Ii_a\subset{\mathfrak B}_\pol\times[-N,N]$, where $\Ii_a:=\pi_n^{-1}(a,0)\cap\pol$. By applying Lemma \ref{gp} (i) to the compact set $K:=\ol{\mathfrak B}_\pol\times[-N,N]$, there exists $M>0$ such that the partial derivative $\frac{\partial\gamma^*_{P,M}}{\partial\x_n}=\frac{\partial\gamma^{\rm p}_{P,M}}{\partial\x_n}$ is strictly positive on $K\cap\{x_n\leq0\}$. On the other hand, if $\pol$ is in weak second trimming position there exists, by Lemma \ref{gp} (ii), $M>0$ such that the partial derivative $\frac{\partial\gamma^*_{P,M}}{\partial\x_n}=\frac{\partial\gamma^{\rm r}_{P,M}}{\partial\x_n}$ is strictly positive on the set $\{x_n\leq0\}$.
\end{substeps}\enlargethispage{0.75mm}

\begin{substeps}{trim2}\label{solofa}
Since $\mathsf{G}_{P,M}^*$ preserves the lines $\pi_n^{-1}(a,0)$, to prove the statement it is enough to show that $\gamma_{P,M}^{*,a}(\Mm_a)=\gamma_{P,M}^{*,a}(\Nn_a)=\Tt_a$ for all $a\in\R^{n-1}$.  Before proving this we point out a couple of facts:
\end{substeps}

\begin{substeps}{trim2}
First,  if $r\leq0$ and $\gamma_{P,M}^{*,a}(r)=r$, then 
\begin{equation}\label{diamond}
\gamma_{P,M}^{*,a}(]{-\infty},r[)={]{-\infty},r[}.
\end{equation}
\end{substeps}
\noindent Indeed, observe that $\gamma_{P,M}^{{\rm p},a}(t)\leq t$ for each $t<0$. Besides, by Lemma~\ref{gp} (ii) the function $\gamma_{P,M}^{{\rm r},a}(t)$ is strictly increasing on the interval ${]{-\infty},0]}$.
Since also $\lim_{t\to{-\infty}}\gamma_{P,M}^{*,a}(t)={-\infty}$ we conclude that equality \eqref{diamond} holds. As a particular case, for $r=0$ we have $\gamma_{P,M}^{*,a}(0)=0$ and so $\gamma_{P,M}^{*.a}(]-\infty, 0[)=]-\infty,0[$.

\begin{substeps}{trim2}\label{diamond1}
Second, if $r\ge 0$ with $rP(r,a)=0$ then $\gamma_{P,M}^{*,a}({]r,+\infty[})={[0,+\infty}[$.
\end{substeps}\\
Indeed, the inclusion $\gamma_{P,M}^{*,a}(]r,+\infty[)\subset[0,+\infty[$ is trivial. To prove the converse, we define 
$$
\delta_a^*(\t):=\begin{cases}
1-\frac{\t P^2(a,\t)}{M}&\text{if $\ast={\rm p}$}\\
1-\t P^2(a,\t)&\text{if $\ast={\rm r}$}\\
\end{cases}
$$ 
and observe that if $rP(a,r)=0$ then $\delta_a^*(r)=1$. Thus, since $\lim_{t\to\infty}\delta_a^*(t)={-\infty}$, there is $t_0>r$ such that $\delta_a^*(t_0)=0$. By continuity, we deduce 
$$
{[0,+\infty[}\subset\gamma_{P,M}^{*,a}([t_0,+\infty[)\subset\gamma_{P,M}^{*,a}(]r,+\infty[)\subset{[0,+\infty[},
$$ 
and so \ref{trim2}.\ref{diamond1} is proved. This applies in particular to $r=0$ and so, for each $a\in\R^{n-1}$ we have
\begin{equation}\label{heart}
\gamma_{P,M}^{*,a}(]0,+\infty[)=[0,+\infty[.
\end{equation}

Let us prove now \ref{trim2}.\ref{solofa}. We distinguish the following cases:
 
\vspace{2mm}
\noindent{\bf Case 1}: $\Mm_a=\R$. Then  $\Mm_a\setminus \{0\}\subset\Nn_a\subset \Mm_a$, and in view of equalities \eqref{diamond} and \eqref{heart} we have
$$
\R=\gamma_{P,M}^{*,a}(]{-\infty},0[)\cup\gamma_{P,M}^{*,a}(]0,+\infty[)\subset\gamma_{P,M}^{*,a}(\Nn_a)\subset\gamma_{P,M}^{*,a}(\Mm_a)\subset\R.
$$
Thus, $\gamma_{P,M}^{*,a}(\Mm_a)=\gamma_{P,M}^{*,a}(\Nn_a)=\R=\Tt_a$.

\vspace{2mm}
\noindent{\bf Case 2}: $\Mm_a=\R\setminus]p,q[$, where ${-\infty}\leq p<q$. Observe that $(a,q)\in\partial\pol$ and, if $p>{-\infty}$, also $(a,p)\in\partial\pol$. The following equalities also hold:
$$
\begin{array}{lll}
\lim_{t\to{-\infty}}\gamma_{P,M}^{*,a}(t)={-\infty},&\quad&\gamma_{P,M}^{*,a}(p)=p\quad\text{(if $p>{-\infty}$)},\\[8pt]
\lim_{t\to+\infty}\gamma_{P,M}^{*,a}(t)=+\infty,&\quad&\gamma_{P,M}^{*,a}(q)=q.
\end{array}
$$

Now, we analyze the behavior of the function $\gamma_{P,M}^{*,a}$ on unbounded intervals of interest:

\noindent (2.a) If $q<0$, then $\gamma_{P,M}^{*,a}([q,+\infty[)=\gamma_{P,M}^{*,a}({[q,+\infty[}\setminus\{0\})=[q,+\infty[$. 

\noindent Indeed, since $\Mm_a\setminus\{0\}\subset \Nn_a\subset \Mm_a$ and $\gamma_{P,M}^{*,a}$ is strictly increasing in the interval $[q,0]\subset[-N,N]$ (see \ref{trim2}.\ref{signder2}), we deduce that $\gamma_{P,M}^{*,a}([q,0[)=[q,0[$. We conclude, using once more equality \eqref{heart}, that 
$$
{[q,+\infty[}={[q,0[}\cup\ {[0,+\infty[}=\gamma_{P,M}^{*,a}({[q,+\infty[}\setminus\{0\})=\gamma_{P,M}^{*,a}([q,+\infty[).
$$

\noindent(2.b) If $q\ge 0$ then $\gamma_{P,M}^{*,a}([q,+\infty[)=\gamma_{P,M}^{*,a}({[q,+\infty[} \setminus\{0\})=[0,+\infty[$, in view of \ref{trim}.\ref{diamond1}.

\noindent(2.c) If ${-\infty}<p<0$ then $\gamma_{P,M}^{*,a}(]{-\infty},p])={]{-\infty},p]}$, by equality \eqref{diamond}. 

\noindent(2.d) Finally, if $p\ge 0$ we have ${]{-\infty},0[}\subset\gamma_{P,M}^{*,a}(]{-\infty},p]\setminus\{0\})\subset\gamma_{P,M}^{*,a}(]{-\infty},p])$, as a consequence of equality \eqref{diamond}. 

Putting all together we obtain 
$$
\gamma_{P,M}^{*,a}(\Mm_a)=\gamma_{P,M}^{*,a}(\Nn_a)=\left\{\begin{array}{lll}
\Mm_a=\Tt_a&\ &\text{if $q<0$,}\\[4pt]
\Mm_a\cup{[0,+\infty[}=\Tt_a&\ &\text{if $q\geq0$,}
\end{array}\right.
$$ 
and the proof of \ref{trim2}.\ref{solofa} is complete.
\end{proof}

\section{Complement of the interior of a convex polyhedron}\label{s7}

In this section we prove the part of Theorem \ref{mainext} concerning the complement of the interior of a convex polyhedron. More precisely, we prove the following:
\begin{thm}\label{redextai}
Let $n\geq2$ and let $\pol\subset\R^n$ be an $n$-dimensional convex polyhedron which is not a layer. We have: 
\begin{itemize}
\item[(i)] If $\pol$ is bounded, then $\R^n\setminus\Int(\pol)$ is a polynomial image of $\R^n$. 
\item[(ii)] If $\pol$ is unbounded, then $\R^n\setminus\Int(\pol)$ is regular image of $\R^n$.
\end{itemize}
\end{thm}

Again, the proof of the previous results runs by induction on the number of facets of $\pol$ in case (i) and by double induction on the dimension and the number of facets of $\pol$ in case (ii). The first step of the induction process in the second case concerns dimension $2$, which once more has already being approached by the second author in \cite{u2} with similar but less involving techniques. We start by approaching a particular case.

\begin{lem}[The orthant]\label{conocerrado} 
Let $\pol\subset\R^n$ be an $n$-dimensional nondegenerate convex polyhedron with $n\geq2$ facets. Then  $\R^n\setminus\Int(\pol)$ is a polynomial image of $\R^n$. 
\end{lem}
\begin{proof}
After a change of coordinates, it is enough to show, by induction on $n$, that $\pp(\R^n\setminus\Qq_n)=n$ for each $n\geq 2$, where 
$$
\Qq_n:=\{x_1>0,\dots,x_n>0\}\subset\R^n.
$$ 
For $n=2$ apply \cite[Thm.1]{u2} and assume, by induction hypothesis, that $\pp(\R^{n-1}\setminus\Qq_{n-1})=n-1$ is true (if $n\geq3$), that is, there is a polynomial map $f:\R^{n-1}\to\R^{n-1}$ such that $f(\R^{n-1})=\R^{n-1}\setminus\Qq_{n-1}$. Define 
$$
f^{(1)}:\R^n\to\R^n,\, (x_1,\dots,x_n)\mapsto(f(x_1,\dots,x_{n-1}),x_n),
$$ 
which satisfies the equality 
$$
f^{(1)}(\R^n)=f(\R^{n-1})\times\R=(\R^{n-1}\setminus\Qq_{n-1})\times\R.
$$ 
Let $f^{(2)}:\R^n\to\R^n$ be a change of coordinates such that $f^{(2)}((\R^{n-1}\setminus\Qq_{n-1})\times\R)=\R^n\setminus\Int(\pol_1)$, where
$$
\pol_1:=\{x_1\ge 0,\dots,x_{n-3}\ge0,\ x_{n-2}+x_n\ge 0,\ x_{n-2}-x_n\ge 0\}\subset\R^n
$$
is a convex polyhedron in strong first trimming position (see the proof of Lemma \ref{lem2}). By Lemma \ref{trim} there exists a polynomial map $f^{(3)}:\R^n\to\R^n$ such that 
$$
f^{(3)}(\R^n\setminus\Int(\pol_1))=\R^n\setminus(\Int(\pol_1)\cap\{x_{n-1}\le 0\}).
$$
Consider now the change of coordinates
$$
f^{(4)}:\R^n\to\R^n,\ (x_1,\ldots,x_n)\mapsto(x_1,\ldots,x_{n-2},x_n,x_{n-1});
$$
notice that $\pol_1'=f^{(4)}(\pol_1)$ is in extreme optimal strong second trimming position with respect to the facet $\Ff':=f^{(4)}(\pol_1\cap{x_{n-1}=0})=\pol_1'\cap\{x_n=0\}$. By Lemma \ref{trim2} there exists a polynomial map $f^{(5)}:\R^n\to\R^n$ such that
$$
f^{(5)}(\R^n\setminus(\Int(\pol_1')\cap\{x_n\leq0\}))=\R^n\setminus(\Int(\pol_1')\cap\{x_n<0\}).
$$
Moreover, there is a change of coordinates $f^{(6)}:\R^n\to\R^n$ such that
$$
f^{(6)}(\R^n\setminus(\Int(\pol_1')\cap\{x_n<0\}))=\R^n\setminus\Qq_n.
$$
Thus, the composition $g:=f^{(6)}\circ f^{(5)}\circ f^{(4)}\circ f^{(3)}\circ f^{(2)}\circ f^{(1)}:\R^n\to\R^n$ is a polynomial map which satisfies $g(\R^n)=\R^n\setminus\Qq_n$.
\end{proof}

\begin{proof}[Proof of Theorem \em\ref{redextai}(i)] 
Suppose first that $\pol:=\Delta$ is an $n$-simplex. We may assume that its vertices are
$$
v_0:=(0,\dots,0),\ v_i:=(0,\dots,0,1^{(i)},0,\dots,0)\ \text{for}\ i\neq n-1,\ \text{and}\ v_{n-1}:=(0,\dots,0,-1,0);
$$
this convex polyhedron is in strong first trimming position because it is bounded. We write $\Delta:=\bigcap_{i=1}^{n+1}{H_{i}'}^+$ where $\{H_1',\dots,H_{n+1}'\}$ is the minimal presentation of $\Delta$ and the hyperplanes $H_i'$ are ordered in such a way that $H_1':=\{x_n=0\}$. Consequently, the convex polyhedron $\Delta_{1,\times}:=\bigcap_{i=2}^{n+1}{H_i'}^+$ is convex, $n$-dimensional, nondegenerate and has $n$ facets. By Lemma \ref{conocerrado} there is a polynomial map $g:\R^n\to\R^n$ such that $g(\R^n)=\R^n\setminus\Int\Delta_{1,\times}$. By Lemmata \ref{trim} and \ref{trim2}, and using suitably the change of coordinates which interchanges the variables $\x_{n-1}$ and $\x_n$ (as we have done above in the proof of Lemma \ref{conocerrado}), there is a polynomial map $h:\R^n\to\R^n$ such that $h(\R^n\setminus\Int\Delta_{1,\times})=\R^n\setminus\Int\Delta$, and therefore the composition $f:=h\circ g$ is a polynomial map such that $f(\R^n)=\R^n\setminus\Int\Delta$.

\vspace{1mm}
Now, let $\pol$ be an $n$-dimensional bounded convex polyhedron of $\R^n$. Let $\Delta$ be an $n$-simplex such that $\pol\subset\Int\Delta$. Write $\pol:=\bigcap_{i=1}^mH_i^+$, where $\{H_1,\dots,H_m\}$ is the minimal presentation of $\pol$, and consider the convex polyhedra $\{\pol_{0},\dots,\pol_{m}\}$ defined by
$$
\pol_0:=\Delta,\quad \pol_j:=\pol_{j-1}\cap H_{j}^+\ \text{for}\ 1\le j\le m.
$$
Thus, $\pol_{m}=\pol$ and for each $1\leq j\leq m$ the convex polyhedron $\pol_{j}=\Delta\cap(\bigcap_{i=1}^jH_i^+)$ has a facet $\Ff_{j}$ contained in the hyperplane $H_j$. Moreover, the set $\R^n\setminus\Delta$ is, as we have just seen above, the image of a polynomial map $f^{(0)}:\R^n\to\R^n$. 

For each index $1\leq j\leq m$ there is a change of coordinates $h^{(j)}:\R^n\to\R^n$ such that $h^{(j)}(H_j^+)=\{x_{n-1}\leq0\}$; denote
$$
\pol_j':=h^{(j)}(\pol_j)=h^{(j)}(\pol_{j-1})\cap h^{(j)}(H_{j}^+)=h^{(j)}(\pol_{j-1})\cap\{x_{n-1}\leq0\}.
$$
By Lemmata \ref{trim} and \ref{trim2}, and using again suitably the change of coordinates that interchange the variables $\x_{n-1}$ and $\x_n$, there is a polynomial map $g^{(j)}:\R^n\to\R^n$ such that
$$
g^{(j)}(\R^n\setminus h^{(j)}(\Int(\pol_{j-1})))=\R^n\setminus h^{(j)}(\Int(\pol_{j-1}\cap \{x_{n-1}\le 0\}))=\R^n\setminus\Int(\pol'_{j});
$$
hence, the polynomial map $f^{(j)}:=(h^{(j)})^{-1}\circ g^{(j)}\circ h^{(j)}:\R^n\to\R^n$ satisfies the equality
$$
f^{(j)}(\R^n\setminus \Int(\pol_{j-1}))=\R^n\setminus\Int(\pol_j).
$$
Thus, the composition $f:=f^{(m)}\circ\cdots\circ f^{(0)}:\R^n\to\R^n$ is a polynomial map which satisfies $f(\R^n)=\R^n\setminus\Int(\pol_m)=\R^n\setminus\Int(\pol)$, as wanted.
\end{proof}

\begin{proof}[Proof of Theorem \em\ref{redextai}(ii)]
We proceed by induction on the pair $(n,m)$ where $n$ is the dimension of $\pol$ and $m$ is the number of facets of $\pol$. As we have already commented the case $n=2$ has already solved in \cite[Thm.1]{u2} (using polynomial maps) and we do not include further details. Suppose in what follows that $n\geq3$ and that the result holds for any $d$-dimensional convex polyhedron of $\R^d$ which is not a layer, if $2\leq d\leq n-1$. We distinguish two cases:

\noindent{\bf Case 1.} \em $\pol$ is a degenerate convex polyhedron\em. This case runs parallel to Case 1 in the proof of Theorem~\ref{redexta}(ii). By Lemma~\ref{lem:deg2}, we may assume that $\pol=\p\times\R^{n-k}$ where $1\le k<n$ and $\p\subset\R^k$ is a nondegenerate convex polyhedron. If $k=1$ then $\pol$ is either a layer, which disconnects $\R^n$ and cannot be a regular image of $\R^n$, or a half-space. It is trivial to check that the complement of the interior of a half-space is indeed a regular (in fact, polynomial) image of $\R^n$. We now assume $k>1$; by the induction hypothesis $\R^k\setminus\Int\p$ is the image of $\R^k$ by a regular map $f:\R^k\to\R^k$; hence, $\R^n\setminus\Int(\pol)$ is the image of the regular map $(f,\id_{\R^{n-k}}):\R^k\times\R^{n-k}\to\R^k\times\R^{n-k}$.

\vspace{1mm}
\noindent{\bf Case 2.} \em $\pol$ is a nondegenerate convex polyhedron\em. If $\pol$ has $n$ facets (and so just one vertex), the result follows from Lemma \ref{conocerrado}. Suppose now that the number $m$ of facets of $\pol$ is strictly greater than $n$ and let $\{H_1,\ldots,H_m\}$ be the minimal presentation of $\pol:=\bigcap_{i=1}^mH_i^+$. We may assume, by Lemma \ref{colocarcasipmn}, that $\pol$ is in extreme weak first trimming position with respect to $\Ff_1$. 

By Lemma \ref{lem:ppr} the convex polyhedron $\pol_{1,\times}:=\bigcap_{j=2}^mH_j^+$ is in weak first trimming position. Thus, we can apply Proposition \ref{trim} to the convex polyhedron $\pol_{1,\times}$, and so there is a regular map $f^{(1)}:\R^n\to\R^n$ such that 
$$
f^{(1)}(\R^n\setminus\Int(\pol_{1,\times}))=\R^n\setminus(\Int(\pol_{1,\times})\cap\{x_{n-1}\le 0\}).
$$
Observe that 
\begin{equation*}
\begin{split}
\Int(\pol_{1,\times})\cap\{x_{n-1}\leq0\}&=(\Int(\pol_{1,\times})\cap\{x_{n-1}<0\})\cup(\Int(\pol_{1,\times})\cap\{x_{n-1}=0\})\\
&=\Int(\pol_{1,\times}\cap\{x_{n-1}\leq0\})\cup \Lambda=\Int(\pol)\cup\Lambda,
\end{split}
\end{equation*}
where 
$$
\Lambda:=\Int(\pol_{1,\times})\cap\{x_{n-1}=0\}\subset\pol_{1,\times}\cap\{x_{n-1}=0\}=\pol\cap\{x_{n-1}=0\}=\Ff_1
$$ 
is a subset of the interior of the facet $\Ff_1$ (as a topological manifold with boundary). Let $f^{(2)}:\R^n\to\R^n$ be a change of coordinates such that $\pol':=f^{(2)}(\pol)\subset\{x_n\leq0\}$ and $\Ff_1':=f^{(2)}(\Ff_1)=\pol\cap\{x_n=0\}$; by Remark \ref{bajuste}, $\pol'$ is in extreme weak second trimming position with respect to the facet $\Ff_1$. Denote $\Lambda':=f^{(2)}(\Lambda)$.

By Lemma \ref{trim2}, there is a regular map $f^{(3)}:\R^n\to\R^n$ such that 
$$
f^{(3)}(\R^n\setminus(\Int(\pol')\cup\Lambda'))=\R^n\setminus\Int(\pol').
$$ 
Moreover, $\pol_{1,\times}$ is a convex polyhedron with $m-1\geq n$ facets; hence, it is not a layer, and by induction hypothesis there is a regular map $g:\R^n\to\R^n$ such that $g(\R^n)=\R^n\setminus\Int(\pol_{1,\times})$. We conclude that the regular map $h:=(f^{(2)})^{-1}\circ f^{(3)}\circ f^{(2)}\circ f^{(1)}\circ g$ satisfies 
\begin{equation*}
\begin{split}
h(\R^n)&=(f^{(2)})^{-1}\circ f^{(3)}\circ f^{(2)}\circ f^{(1)}(\R^n\setminus\pol_{1,\times})=
(f^{(2)})^{-1}\circ f^{(3)}\circ f^{(2)})(\R^n\setminus(\Int(\pol)\cup\Lambda))\\
&=(f^{(2)})^{-1}\circ f^{(3)})(\R^n\setminus(\Int(\pol')\cup\Lambda'))=(f^{(2)})^{-1}(\R^n\setminus\Int(\pol'))
=\R^n\setminus\Int(\pol),
\end{split}
\end{equation*}
and we are done. 
\end{proof}

\begin{remark}
Again, the previous proof is constructive. Moreover, if the convex polyhedron that appears in a certain step of the inductive process is in strong first trimming position with respect to a facet $\Ff_1$ and it can also be placed in strong second trimming position with respect to that facet, then combining Lemmata \ref{trim} and \ref{trim2} the regular map of this step can be chosen polynomial. In view of Definitions \ref{1tm} and \ref{2tm}, this reduces the complexity of the final regular map. Moreover, if this can be done in each inductive step, the final regular map can be assumed to be polynomial.
\end{remark}

\section{Application: Exterior of the $n$-dimensional ball}\label{s8}

The complement of a closed ball can be understood as the limit of the complements of a suitable family of bounded convex polyhedra (consider triangulations of the closed ball) and it is natural to wonder if this semialgebraic set is also a polynomial image of $\R^n$. However, this complement is never a polynomial image of $\R^n$. Indeed, consider the $n$-dimensional closed ball $\ol{\Bb}_n({\bf 0},1)$ of radius $1$ centered at the origin of $\R^n$ and its complement $\Ss:=\R^n\setminus\ol{\Bb}_n({\bf 0},1)$. Consider the $(n-1)$-dimensional algebraic set 
$$
X:=\partial\Ss=\ol{\Ss}\setminus\Ss=\{(x_1,\ldots,x_n)\in\R^n:\ x_1^2+\cdots+x_n^2=1\}.
$$
Since $X\cap\ol{\Ss}=X$ is a bounded set and $\dim(X\cap\partial \Ss)=n-1$, we deduce by \cite[3.4]{fg2}, that $\Ss$ is not polynomial image of $\R^n$.

Nevertheless, we prove here that the complement of the open ball of $\R^n$, which is the limit of the complements of the interiors of a suitable family of bounded convex polyhedra, is a polynomial image of $\R^n$. The $2$-dimensional case was approached in \cite[4.1]{fg2} with a very specific proof which cannot be generalized to the $n$-dimensional case. Our proof here is based on the results obtained in the precedent sections concerning complements of convex polyhedra. As before, the proofs in this section assume that $n\ge 2$. 

\begin{prop}\label{openball}
The complement of an open ball of $\R^n$ is a polynomial image of $\R^n$.
\end{prop}
\begin{proof} 
It is enough to prove that the complement of the open ball $\Bb_n({\bf 0},1)\subset\R^n$ of radius $1$ centered at the origin is a polynomial image of $\R^n$. To that end, consider the $n$-dimensional cube $\Cc:=[-1,1]^n$ and recall that, by Theorem \ref{mainext}, there is a polynomial map $f^{(1)}:\R^n\to\R^n$ whose image is $\R^n\setminus\Cc$. Observe also that $\Bb_n({\bf 0},1)\subset\Cc\subset\Bb_n({\bf 0},n)$.

\vspace{2mm}\setcounter{substep}{0}
\begin{substeps}{openball}
Next, we construct a biquadratic polynomial $g:=g_n\in\R[{\t}^2]$ such that the polynomial $h:=h_n:=\t g$ defines an increasing function on the interval $[n,+\infty[$ and satisfies the following conditions: 
$$
h(0)=0,\ h(1)=h(n)=1,\ h'(n)=0,\ h(]1,n[)\subset[1,+\infty[\quad \text{and}\quad h([0,+\infty[)=[0,+\infty[.
$$ 
This implies, in particular, that $h(]1,+\infty[)=h([n,+\infty[)=[1,+\infty[$ and the following equality needed later:
\begin{equation}\label{interpol}
h([\mu,\infty[)=[1,+\infty[\quad\text{for each}\quad 1\leq\mu\leq n.
\end{equation}

Next, let us check that the polynomial $g(\t):=a_n\t^4+b_n\t^2+c_n$, where
$$
a_n:=\dfrac{1+2n}{2n^3(1+n)^2},\quad
b_n:=-\dfrac{1+2n+3n^2+4n^3}{2n^3(1+n)^2}\quad\text{and}\quad
c_n:=\dfrac{3+6n+4n^2+2n^3}{2n(1+n)^2},
$$
satisfies the required conditions. The values $a_n$, $b_n$ and $c_n$ are obtained by imposing conditions
$h(1)=h(n)=1$ and $h'(n)=0$ to a polynomial of type $a_n\t^5+b_n\t^3+c_n\t$.
\end{substeps}

\vspace{1mm}
Indeed, $h(0)=0$ and the derivative $h'(\t)=5a_n\t^4+3b_n\t^2+c_n$ of $h$ in $\t=1$ becomes
$$
h'(1)=\frac{n^4+n^3-4n^2+n+1}{(1+n)n^3}=\frac{(n-1)^2(n^2+3n+1)}{(n+1)n^3}>0;
$$
hence, $h$ is increasing on a neighborhood of $\t=1$. Observe that $a_n,c_n>0$ and $b_n<0$, and so by Descartes' chain rule (see \cite[1.2.14]{bcr}), the number of positive roots of $h'(\t)$ counted with their multiplicity is at most $2$. By Rolle's Theorem we deduce using equalities $h(1)=h(n)=1$ that there is $\xi\in{]1,n[}$ with $h'(\xi)=0$. Since $h'(n)=0$, the derivative $h'(\t)$ has exactly two positive roots, which are $\xi$ and $n$. This implies, since $a_n>0$, that the derivative $h'(\t)$ is strictly positive on the interval $]n,+\infty[$; consequently, $h$ is an increasing function on the interval $[n,+\infty[$. 

Let us show now that $h(]1,n[)\subset[1,+\infty[$. Since $h$ is an increasing function on a neighborhood of $\t=1$ and since $h(1)=h(n)=1$ and $h'(\t)$ has exactly one root in the interval $]1,n[$, we deduce that $h(\t)>1$ for all $\t\in{]1,n[}$. Otherwise there would exist $\t_0\in{]1,n[}$ such that $h(\t_0)=1$ and, by Rolle's Theorem, there would be $\xi_1\in{]1,\t_0[}$ and $\xi_2\in{]\t_0,n[}$ such that $h'(\xi_i)=0$ for $i=1,2$, a contradiction.

Finally, we prove the equality $h([0,+\infty[)=[0,+\infty[.$ Indeed, $h$ is an increasing function on the interval $]0,\xi[$, because $h'(\t)$ has no roots in $]0,\xi[$ and $h'(1)>0$. Since $h(\xi)>1=h(n)$ and $h'(\t)$ does not vanish in the interval $]\xi,n[$, the function $h$ is decreasing in such interval; hence, $h([0,n])=[0,h(\xi)]$. On the other hand, $h(\t)$ is increasing in the interval $[n,+\infty[$ and so $h([n,+\infty[)=[h(n),+\infty[$. Consequently, since $h(n)\leq h(\xi)$ it holds
$$
h([0,+\infty[)=h([0,n])\cup h([n,+\infty[)=[0,h(\xi)]\cup[h(n),+\infty[\ =[0,+\infty[.
$$
\vspace{2mm}
\begin{substeps}{openball}
Next, since $g\in\R[{\t}^2]$, the map 
$$
f^{(2)}:\R^n\to\R^n,\ x:=(x_1,\ldots,x_n)\mapsto g(\|x\|)(x_1,\ldots,x_n)
$$
is polynomial. Moreover, if as usual $\sph^{n-1}:=\{\vec{v}\in\R^n:\,\|\vec{v}\|=1\}$, for each $\vec{v}\in\sph^{n-1}$ we have
$$
f^{(2)}(t\vec{v})=g(\|t\vec{v}\|)t\vec{v}=g(t)t\vec{v}=h(t)\vec{v}\ \forall\,t\in[0,+\infty[\ \Longrightarrow\ \|f^{(2)}(t\vec{v})\|=h(t)\ \forall\,t\in[0,+\infty[.
$$
Let $\mu_{\vec{v}}\in{]1,n[}$ be such that
$$
{\bf0}\vec{v}\cap(\R^n\setminus\Cc)=\{{\bf0}+t\vec{v}:\, t\in[\mu_{\vec{v}},+\infty[\};
$$
we deduce from equality \eqref{interpol} that $f^{(2)}({\bf0}\vec{v}\cap(\R^n\setminus\Cc))=\{{\bf0}+t\vec{v}:\,t\in[1,+\infty[\}$. Consequently, 
$$
f^{(2)}(\R^n\setminus\Cc)=\bigcup_{\vec{v}\in\sph^{n-1}}f^{(2)}({\bf0}\vec{v}\cap(\R^n\setminus\Cc))=\bigcup_{\vec{v}\in\sph^{n-1}}\{{\bf0}+t\vec{v}:\, t\in[1,+\infty[\}=\R^n\setminus\Bb_n({\bf 0},1).
$$
\end{substeps}
Finally, the polynomial map $f:=f^{(2)}\circ f^{(1)}:\R^n\to\R^n$ satisfies 
$$
f(\R^n)=f^{(2)}(f^{(1)}(\R^n))=f^{(2)}(\R^n\setminus\Cc)=\R^n\setminus\Bb_n({\bf 0},1),
$$ 
as wanted.
\end{proof}

The complement of a closed ball $\ol{\Bb}$ in $\R^n$ is not a polynomial image of $\R^n$, but we prove next that it is however a polynomial image of $\R^{n+1}$; hence, ${\rm p}(\R^n\setminus\ol{\Bb})=n+1$. 

\begin{cor}\label{closed ball} 
The complement of a closed ball $\ol{\Bb}$ of $\R^n$ is a polynomial image of $\R^{n+1}$.
\end{cor}
\begin{proof}
It is enough to prove that $\R^n\setminus\ol{\Bb}_n({\bf 0},1)$ is a polynomial image of $\R^{n+1}$. By Proposition \ref{openball}, there is a polynomial map $g^{(1)}:\R^n\to\R^n$ whose image is $\R^n\setminus\Bb_n({\bf 0},1)$. On the other hand, by \cite[1.4 (iv)]{fg1}, the image of the polynomial map
$$
g^{(2)}:\R^2\to\R^2,\ (x,y)\mapsto((xy-1)^2+x^2,\ y(xy-1))
$$
is the open half-space $]0,+\infty[\times\R$. Consider now the polynomial maps
$$
\begin{array}{rl}
f^{(1)}:\R^{n+1}\to\R^n,&\ (x_1,\ldots,x_{n+1})\mapsto (1+x_1)g^{(1)}(x_2,\ldots,x_{n+1}),\\[4pt]
f^{(2)}:\R^{n+1}\to\R^{n+1},&\ (x_1,\ldots,x_{n+1})\mapsto (g^{(2)}(x_1,x_2),x_3,\ldots,x_{n+1}).
\end{array}
$$
Since $f^{(2)}(\R^{n+1})={]0,+\infty[}\times\R^n$, the polynomial map $f:=f^{(1)}\circ f^{(2)}:\R^{n+1}\to\R^n$ satisfies
$$
f(\R^{n+1})=f^{(1)}(f^{(2)}(\R^{n+1}))=f^{(1)}(]0,+\infty[\times\R^n)=\R^n\setminus\ol{\Bb}_n({\bf 0},1),
$$ 
as wanted.
\end{proof}

\bibliographystyle{amsalpha}

\end{document}